\documentclass[a4paper,11pt,leqno]{article}

\usepackage[latin1]{inputenc}
\usepackage[T1]{fontenc} 
\usepackage[english]{babel}
\usepackage[notcite,notref]{showkeys}
\usepackage{verbatim}
\usepackage{graphics}
\usepackage{amsfonts}
\usepackage{amsmath}
\usepackage{amsthm}
\usepackage{amssymb}
\usepackage{color}
\usepackage{graphicx}
\usepackage{bm}

\theoremstyle{definition}
\newtheorem{defin}{Definition}[section]
\newtheorem{ex}[defin]{Example}
\newtheorem{rem}[defin]{Remark}

\theoremstyle{plane}
\newtheorem{thm}[defin]{Theorem}
\newtheorem{prop}[defin]{Proposition}
\newtheorem{coroll}[defin]{Corollary}
\newtheorem{lemma}[defin]{Lemma}

\newcommand{\mbb}{\mathbb}
\newcommand{\mbf}{\mathbf}
\newcommand{\mc}{\mathcal}
\newcommand{\veps}{\varepsilon}
\newcommand{\what}{\widehat}
\newcommand{\wtilde}{\widetilde}
\newcommand{\vphi}{\varphi}
\newcommand{\oline}{\overline}

\newcommand{\ra}{\rightarrow}

\newcommand{\hra}{\hookrightarrow}

\newcommand{\longra}{\longrightarrow}
\newcommand{\g}{\gamma}
\newcommand{\vrho}{\varrho}
\newcommand{\lan}{\langle}
\newcommand{\ran}{\rangle}

\newcommand{\R}{\mathbb{R}}
\newcommand{\C}{\mathbb{C}}
\newcommand{\N}{\mathbb{N}}
\newcommand{\Z}{\mathbb{Z}}
\newcommand{\z}{\zeta}

\renewcommand{\Re}{{\rm Re}\,}

\newcommand{\Id}{{\rm Id}\,}

\allowdisplaybreaks

\def\d{\partial}

\title{\Large{\bfseries{\textsc{On the Cauchy problem for microlocally symmetrizable \\ \vspace{.2cm} hyperbolic systems 
with log-Lipschitz coefficients}}}}

\author{\normalsize\textsl{Ferruccio Colombini}$\,^1\;$, \textsl{Daniele Del Santo}$\,^2\;$, \textsl{Francesco Fanelli}$\,^3\;$,
\textsl{Guy M\'etivier}$\,^{4}$ \vspace{.5cm} \\
\footnotesize{$\,^1\;$ \textsc{Universit\`a di Pisa}} \\ {\footnotesize Dipartimento di Matematica} \\
\footnotesize{\ttfamily{colombini@dm.unipi.it}} \vspace{0.3cm} \\
\footnotesize{$\,^2\;$ \textsc{Universit\`a di Trieste}} \\ {\footnotesize Dipartimento di Matematica e Geoscienze} \\
\footnotesize{\ttfamily{delsanto@units.it}} \vspace{0.3cm} \\
\footnotesize{$\,^3\;$ \textsc{Universit\'e de Lyon, Universit\'e Claude Bernard Lyon 1}} \\
{\footnotesize Institut Camille Jordan, UMR 5208} \\
\footnotesize{\ttfamily{fanelli@math.univ-lyon1.fr}} \vspace{0.3cm} \\
\footnotesize{$\,^4\;$ \textsc{Universit\'e de Bordeaux}} \\ {\footnotesize Institut de Math\'ematiques de Bordeaux, UMR 5251} \\
\footnotesize{\ttfamily{guy.metivier@math.u-bordeaux1.fr}}}

\vspace{.2cm}
\date\today

\begin{document}

\maketitle

\subsubsection*{Abstract}
{\footnotesize The present paper concerns the well-posedness of the Cauchy problem for microlocally symmetrizable hyperbolic systems
whose coefficients and symmetrizer are log-Lipschitz continuous, uniformly in time and space variables. For the global in space problem
we establish energy estimates with finite loss of derivatives, which is linearly increasing in time. This implies well-posedness
in $H^\infty$, if the coefficients enjoy enough smoothness in $x$. From this result, by standard arguments (i.e. extension and convexification) we deduce also local existence
and uniqueness. A huge part of the analysis is devoted to give an appropriate sense to the
Cauchy problem, which is not evident a priori in our setting, due to the very low regularity of coefficients and solutions.}

\paragraph*{2010 Mathematics Subject Classification:}{\small 35L45 (primary); 35B45, 35B65 (secondary).}

\paragraph*{Keywords:}{\small hyperbolic systems, microlocal symmetrizability, log-Lipschitz regularity,
loss of derivatives, global and local Cauchy problem, well-posedness.}

\section{Introduction}

In the present paper we study local and global questions related to the well-posedness of the Cauchy problem for $m\times m$
first order hyperbolic systems
\begin{equation} \label{intro_eq:L}
Lu(t,x)\,:=\,A_0(t,x)\,\d_tu(t,x)\,+\,\sum_{j=1}^n A_j(t,x)\,\d_ju(t,x)\,+\,B(t,x)\,u(t,x)
\end{equation}
under low regularity assumptions on its coefficients.

In \cite{Iv-Pet}, Ivri\u{\i} and Petkov proved that a necessary condition for the well-posedness in the energy space $L^2(\R^n)$
is the existence of a bounded microlocal symmetrizer $S(t,x,\xi)$ for $L$ (see Definition \ref{d:micro_symm} below).
We remark that this is equivalent to the hypothesis of strong hyperbolicity of the operator (see \cite{M-2014} and the references therein).
Nonetheless, this condition is far from being sufficient, even for $\mc{C}^\infty$ well-posedness: see counterexamples in
\cite{Strang}, \cite{M-2014} and \cite{C-M_2015}.

On the other hand, in the simplest case when the first order coefficients $A_j$ are symmetric matrices, i.e. the system is symmetric in
the sense of Friedrichs, if they are also Lipschitz continuous over $[0,T]\times\R^n$, energy estimates
(and then well-posedness) in $L^2$ are quite easy to obtain. Notice that, in this instance, the symmetrizer is simply the identity matrix, and it is
constant (and then smooth) in $(t,x,\xi)$.

Keeping the $W^{1,\infty}([0,T]\times\R^n)$ regularity for the coefficients of the principal part, the previous result was extended
by M\'etivier in \cite{M-2008} to the case of microlocally symmetrizable systems, under the assumption that
the symmetrizer $S$ is Lipschitz in $(t,x)$ and smooth with respect to $\xi$.
More recently, in \cite{M-2014} M\'etivier also proved that, under a stronger $W^{2,\infty}([0,T]\times\R^n)$ condition on the $A_j$'s,
the existence of a Lipschitz continuous symmetrizer \emph{in all the variables} $(t,x,\xi)$ is sufficient for proving energy estimates
in $L^2$. Whether the additional regularity for the coefficients of $L$ is merely a technical requirement or rather a necessary
hypothesis is not clear at present.

Indeed, patological phenomena may be produced by the lack of suitable regularity of the coefficients. To explain this assertion,
let us make a brief parallel with scalar strictly hyperbolic operators of second order
\begin{equation} \label{intro_eq:wave}
 Wu(t,x)\,:=\,\d_t^2u(t,x)\,-\,\sum_{j,k=1}^n\d_j\Bigl(a_{jk}(t,x)\,\d_ku(t,x)\Bigr)\,.
\end{equation}
It is well-known that, if the coefficients $a_{jk}$ are Lipschitz continuous in $t$ and just bounded in $x$, then the Cauchy problem for $W$
is well-posed in the energy space $H^1\times L^2$. On the contrary, whenever the Lipschitz regularity in $t$ is not met,
additional smoothness in $x$ is required, and the Cauchy problem is no more well-posed in the usual sense. Namely, the solution
loses regularity in the evolution, and energy estimates are recovered in weaker spaces $H^s\times H^{s-1}$, for $|s|<1$, 
which deteriorates with the passing of time: actually, $s=s(t)$, with $s'(t)<0$. Therefore, this justifies the additional regularity in
$x$, which is not just a technical requirement to make energy estimates work, but which is really needed to give sense to the Cauchy problem.

Many are the relevant papers on this subject: see e.g. \cite{C-DG-S}, \cite{C-L}, \cite{C-M}, \cite{Tar}, \cite{C-DS}. We refer to
\cite{C-DS-F-M_tl} and \cite{C-DS-F-M_wp} for an overview and recent results. In passing, we mention that this loss of smoothness
produces relevant effects also at the level of the control of waves. 
In addition, we have to point out that similar phenomena were proved to occur also for transport equations with non-regular coefficients (see e.g.
\cite{B-C_1994}, \cite{D_2005}): we refer e.g. to Chapter 3 of \cite{B-C-D} for a review of previous results in this direction and
for further references.

There is an important feature to point out about the wave operator \eqref{intro_eq:wave}. The work of Tarama \cite{Tar} showed that Zygmund type conditions in time
are well-adapted to this kind of analysis. These are smoothness assumptions which are made on the second variation of the function,
i.e. on the symmetric difference $|f(t+\tau)+f(t-\tau)-2f(t)|$, rather than on its modulus of continuity; besides, they can be related with special Besov type regularities.
It turns out that these conditions are weaker than the corresponding ones made on the first variation, namely on $|f(t+\tau)-f(t)|$:
in particular, one can recover well-posedness (with no loss) for the pure Zygmund condition
$$
|f(t+\tau)+f(t-\tau)-2f(t)|\,\leq\,C\,|\tau|\,, \leqno(Z)
$$
which is weaker than the Lipschitz one. In order to deal with this worse behaviour in hyperbolic Cauchy problems,
one has to introduce a lower order corrector in the definition of the energy: this additional term is necessary to produce special algebraic
cancellations in the estimates, erasing bad remainders arising in the time derivative of the energy. We refer to the above mentioned
works \cite{C-DS}, \cite{C-DS-F-M_tl}, \cite{C-DS-F-M_wp} (where the coefficients depend also on $x$) for further progress in this direction.

\medbreak
Let us come back to the operator $L$, defined in \eqref{intro_eq:L}. In the present paper we aim at investigating the problem of its
well-posedness looking at minimal regularity conditions  in time and space variables for coefficients and symmtrizers:
namely, we will consider the case of non-Lipschitz dependence on $(t,x)$, keeping however the smoothness of $S(t,x,\xi)$ with respect
to $\xi$ (see the discussion above).  Indeed, our analysis makes a broad use of paradifferential calculus, which requires
smoothness of symbols in the dual variable.

In a preliminary study (see \cite{C-DS-F-M_Z-sys}), we focused on time-dependent coefficients: $A_j(t)$ and $B=0$ for
simplicity. Inspired by \cite{Tar}, it was natural to formulate Zygmund type hypotheses on them; on the other hand, in this context
it is out of use to know a priori the existence of a microlocal symmetrizer $S(t,\xi)$ with the same regularity in time.
As a matter of fact, analogoulsy to the case of the wave equation, the time derivative
of the energy $E(t)\,\sim\,\bigl(Su,u\bigr)_{L^2}$ produces bad remainders, and one needs to introduce correctors to cancel them out:
the challenge there was to build up a \emph{suitable} microlocal symmetrizer, for which energy estimates work well.
Hence, we supposed the system to be hyperbolic with constant multiplicities, which means that all the eigenvalues of the principal symbol
are real, semi-simple and with constant multiplicities in $t$ and $\xi$; this condition implies in particular strong hyperbolicity.

In the end, in \cite{C-DS-F-M_Z-sys} we proved well-posedness in any $H^s$ for Zygmund type assumptions (even of integral type), while
energy estimates with a finite time-increasing loss of derivatives for log-Zymgund type assumptions (in the right-hand side of $(Z)$
an extra logarithmic factor $|\log\tau|$ is added), which entail well-posedness
just in $H^\infty$.

We have to point out that the hypothesis of dependence of the $A_j$'s just on time was crucial in \cite{C-DS-F-M_Z-sys},
in order to construct a good symmetrizer, and it is far to be clear at present how to deal with the more general case of dependence also on $x$.

As a first step in this direction, in the present paper we consider coefficients $A_j(t,x)$ with slightly better regularity conditions
in $t$, and non-zero matrices $B(t,x)$: we will see that, now, the presence of $0$-th order term makes some differences in the analysis.
More precisely, we assume $B$ to be $L^\infty\bigl([0,T];\mc{C}^\g(\R^n)\bigr)$,
for some H\"older exponent $0<\g<1$, and the $A_j$'s to be bounded and uniformly log-Lipschitz continuous in their variables.
\begin{defin} \label{d:LL}
Let $\Omega\subset\R^N$ be an open domain.
A function $f\in L^\infty(\Omega)$ is said to be \emph{log-Lipschitz} in $\Omega$, and we write $f\in LL(\Omega)$, if the quantity
$$
|f|_{LL}\,:=\,\sup_{y,z\in\Omega,\,|y-z|<1}
\left(\frac{\left|f(y)\,-\,f(z)\right|}{|y-z|\,\log\bigl(1\,+\,1/|y-z|\bigr)}\right)\,<\,+\infty\,.
$$
We define $\|f\|_{LL}\,:=\,\|f\|_{L^\infty}\,+\,|f|_{LL}$.
\end{defin}

The same regularity hypothesis is assumed in $(t,x)$ also for the microlocal symmetrizer (of course, it is taken smooth
in $\xi$). Indeed, as we will see, corrector terms are no more needed in this case, and the energy can be defined in a classical way,
in terms (roughly speaking) of the $L^2$ scalar product with respect to $S$. Then, the microlocal simmetrizability assumption will be
still suitable for our purposes.

For operator $L$, supplemented with these additional hypotheses, in Theorem \ref{th:en_LL} we establish energy estimates on the whole $[0,T]\times\R^n$ with
time-increasing loss of derivatives, which are the exact couterpart of similar inequalities for the wave operator \eqref{intro_eq:wave}
(see \cite{C-L}, \cite{C-M}). These estimates, however, hold true in low regularity Sobolev spaces $H^s$, for $0<s<\g$, due to the weak
smoothness of the coefficients in $x$. For the same reason, the result is just local in time: if the loss of derivatives is too high, 
$u$ ends up in very weak classes, for which the product with log-Lipschitz or H\"older functions is no more well-defined.

Let us explain better this point. One has to remark (see also Proposition \ref{p:Hol-Sob} and Corollary \ref{c:LL-H^s} below) that
multiplication by a log-Lipschitz function is a self-map of $H^s$ if and only if $|s|<1$, so that the first order part
makes sense in $H^{s-1}$ if and only if $-1<s-1<0$. On the other hand, the $0$-th order term will be treated as a remainder, and so
we need it to be in $H^s$: this is true for $|s|<\g$. Notice here that the argument is symmetric for the conservative counterpart
$$
L^*u\,=\,-\,\d_tu\,-\,\sum_{j=1}^n\d_j\bigl(A_j^*\,u\bigr)\,+\,B^*\,u\,,
$$
and in fact, for $L^*$ one gets a sort of ``dual'' version of the previous result,
in $H^s$ classes for $-\g<s<0$ (but the loss is always increasing in time, of course).

Our result strongly relies on paradifferential calculus in logarithmic Sobolev classes, and especially on a fine analysis of paradifferential
operators associated to symbols which are log-Lipschitz (or H\"older) continuous in $x$, for which we develop also a symbolic calculus.
This analysis allows us to approximate the principal part $\mc{A}$ of $L$ by its paradifferential operator $T_{\mc A}$, and to define an
approximate symmetrizer for $T_{\mc A}$ by taking, roughly speaking, the paradifferential operator associated to the original symmetrizer
$S$ (up to small modifications, required in order to deal with low frequencies).

On the other hand, the time dependence comes into play as well, and the weak smoothness in $t$ causes some troubles in view of energy
estimates. Then, following the approach initiated in \cite{C-DG-S} by Colombini, De Giorgi and Spagnolo for the scalar wave equation, we
need to introduce a regularization in time, and to link the regularization parameter (say) $\veps$ with the dual variable $\xi$.
More precisely, given a simbol $a=a(t,x,\xi)$, in a first moment we smooth it out by convolution with a smoothing kernel: we define the family
$a_\veps=\rho_\veps\,*_t\,a$. Then, in view of closing energy estimates, we have to make the key choice $\veps=1/|\xi|$: this means
that the approximation is different, depending on the size of the frequencies we are looking at.
Therefore, the previous paradifferential calculus construction has to be completely revisited and adapted to treat new symbols
$\wtilde{a}(t,x,\xi)$ which are obtained by the family $\bigl(a_\veps\bigr)_\veps$ performing the choice $\veps=|\xi|^{-1}$.

Let us point out that, in our analysis, the log-Lipschitz continuity in $(t,x)$ of \emph{both} the first order coefficients and the
symmetrizer is exploited in a fundamental way. Also, the counterexamples established in \cite{M-2014} and \cite{C-M_2015} imply somehow
the sharpness of our result. In particular, in \cite{C-M_2015} Colombini and M\'etivier were able to exhibit explicit examples of
$2\times2$ microlocally symmetrizable systems with smooth time-dependent coefficients, for which the following phenomena occur:
if the symmetrizer is $\omega$-continuous for some modulus of continuity $\omega$ which is even slightly worse than Lipschitz, a loss
of derivatives has to occur in the energy estimates (then ill-posedness in $L^2$); if $\omega$ is worse than log-Lipschitz, then
the loss is in general infinite (which shows ill-posedness of the Cauchy problem in $\mc{C}^\infty$).

This having been done, we turn our attention to local in space existence and uniqueness questions. Indeed, our regularity hypotheses
are invariant under smooth change of variables, and thus they are suitable for local analysis. On the other hand, we will prove that
also the microlocal symmetrizability assumption (reformulated in Definition \ref{d:full-symm} in a coordinate
independent way) is invariant under change of variables. So, in Theorems \ref{t:local_e} and \ref{t:local_u} we show
respectively local existence and uniqueness of solutions to the Cauchy problem for our operator $L$ (written in a coordinate independent
way, see formula \eqref{eq:def_P} below). These results are the analogue of what established in \cite{C-M}
for wave operators $W$. In passing, we mention that local uniqueness will be derived from a result about propagation of zero across space-like manifolds,
which also implies finite speed of propagation and a sharp description of the propagation of supports by standard arguments (see e.g. \cite{J-M-R_2005}, \cite{Rauch_2005}).

Of course, since we are in a low regularity framework, the sense of the local Cauchy problem is not clear a priori: therefore, the initial
efforts (see Subection \ref{ss:sense}) will be devoted to explain the setting, how one has to interpret the Cauchy problem
under our assumptions and in which sense one can aim at solving it. 
Then, by a change of variables we will reconduct the analysis to operator $L$ written in the form \eqref{intro_eq:L} above,
and, by a classical convexification argument, we will be able to reduce the proof of the local statements to the global in space
results, previously established.

\medbreak
Before going on, let us give a brief overview of the paper.

In the next section we collect our hypotheses, and we state our main results, first for the global Cauchy problem, and then for the local one.
Section \ref{s:tools} is devoted to Littlewood-Paley analysis of log-Lipschitz and Sobolev classes, and to developing paradifferential calculus
associated to low regularity symbols. This analysis will be the key to the proof of the global statement, which will be presented in Section
\ref{s:en-est}. Finally, in Section \ref{s:local} we discuss the local questions.

\subsection*{Notations}
Before going on, let us introduce some notations.

First of all, given two vectors $v$ and $w$ in $\C^m$, we will denote by $v\cdot w$ the usual hermitian product in $\C^m$ and
by $|v|$ the usual norm of a vector in $\C^m$:
$$
v\,\cdot\,w\,=\,\sum_{j=1}^m v_j\,\oline{w_j}\qquad\mbox{ and }\qquad
|v|^2\,=\,v\,\cdot\,v\,.
$$

On the contrary, given a infinite-dimensional Banach space $X$, we will denote by $\|\,\cdot\,\|_{X}$ its norm and,
if it's Hilbert, by $(\,\cdot\,,\,\cdot\,)_{X}$ its scalar product. Tipically, for us $X=L^2(\R^n;\R^m)$ or $H^s(\R^n;\R^m)$.

We will also set $\mc{M}_m(\C)$ the set of all $m\times m$ matrices whose components are complex numbers, and we will denote
by $|\,\cdot\,|_{\mc{M}}$ its norm:
$$
|A|_{\mc{M}}\,:=\,\sup_{|v|=1}|Av|\,\equiv\,\sup_{|v|\leq1}|Av|\,\equiv\,\sup_{v\neq0}\frac{|Av|}{|v|}\,.
$$

\subsubsection*{Acknowledgements}

The first two authors are members of the Gruppo Nazionale per l'Analisi Matematica, la Probabilit\`a
e le loro Applicazioni (GNAMPA) of the Istituto Nazionale di Alta Matematica (INdAM).

\section{Basic definitions and main results} \label{s:results}

We state here our main results. Let us start by the global in space questions.

\subsection{The global Cauchy problem}
For $m\geq1$, let us consider the $m\times m$ linear first order system
\begin{equation} \label{def:L}
Lu(t,x)\,=\,\d_tu(t,x)\,+\,\sum_{j=1}^nA_j(t,x)\,\d_ju(t,x)\,+\,B(t,x)\,u(t,x)
\end{equation}
defined on a strip $[0,T]\times\R^n$, for some time $T>0$ and $n\geq1$.
We suppose $u(t,x)\in\R^m$ and, for all $1\leq j\leq n$, the matrices $A_j(t,x)\in\mc{M}_m(\C)$ as well as $B(t,x)$.

We define the principal symbol $\mc A$ associated to the operator $L$: for all $(t,x,\xi)\in[0,T]\times\R^n\times\R^n$,
\begin{equation} \label{def:symbol}
\mc{A}(t,x,\xi)\,:=\,\sum_{j=1}^n\xi_j\,A_j(t,x)\,.
\end{equation}
Then, for all $(t,x,\xi)$ fixed, $\mc{A}(t,x,\xi)$ is an $m\times m$ matrix which has complex-valued coefficients.
We denote by $\bigl(\lambda_k(t,x,\xi)\bigr)_{1\leq k\leq m}\subset\C$ its eigenvalues at any point $(t,x,\xi)$.

Analogously, we consider also its conservative counterpart, i.e. the operator $\wtilde{L}$ defined by
\begin{equation} \label{def:L*}
\wtilde{L}u(t,x)\,=\,\d_tu(t,x)\,+\,\sum_{j=1}^n\d_j\bigl(A_j(t,x)\,u(t,x)\bigr)\,+\,B(t,x)\,u(t,x)
\end{equation}
and whose principal symbol is still $\mc{A}(t,x,\xi)$, with a different quantization.

Let us state now our working hypothesis.
First of all, in the sequel we will always assume that our operator $L$ is \emph{hyperbolic}, namely
$\bigl(\lambda_k(t,x,\xi)\bigr)_{1\leq k\leq m}\subset\R$.
As for the coefficients of $L$, we always suppose boundedness: we suppose that, for all
$M\,\in\,\left\{A_1\ldots A_n,B\right\}$,
\begin{equation} \label{hyp:bound}
\bigl\|M\bigr\|_{L^\infty([0,T]\times\R^n;\mc{M}_m(\C))}\,:=\,\sup_{(t,x)}\bigl|M(t,x)\bigr|_{\mc{M}}\,\leq\,K_0\,.
\end{equation}
Concerning their regularity, let us start by focusing on first order coefficients: we assume that, for all $1\leq j\leq n$, the matrix-valued
functions $A_j$ are \emph{log-Lipschitz} continuous ($LL$ in brief) in their variables: there exists
a constant $K_1>0$ such that, for all $\tau>0$ and all $y\in\R^n\setminus\{0\}$, one has
\begin{equation} \label{hyp:LL}
\sup_{(t,x)}\bigl|A_j(t+\tau,x+y)\,-\,A_j(t,x)\bigr|_{\mc{M}}\,\leq\,
K_1\,\bigl(|\tau|+|y|\bigr)\,\log\left(1\,+\,\frac{1}{|\tau|+|y|}\right)\,.
\end{equation}
Notice that, since we are in finite dimension, this is equivalent to require the same condition on each component of $A_j$.

Concerning the coefficient of the lower order term, we will assume $\g$-H\"older continuity in space, uniformly in time. More precisely,
we suppose that there exist a $\g\in\,]0,1[\,$ and a constant $K_2>0$ such that, for all $y\in\R^n\setminus\{0\}$, one has
\begin{equation} \label{hyp:Holder}
\sup_{(t,x)}\bigl|B(t,x+y)\,-\,B(t,x)\bigr|_{\mc{M}}\,\leq\,K_2\,|y|^\g\,.
\end{equation}

Finally, we will require that the system is \emph{uniformly microlocally symmetrizable}, in the sense of M\'etivier (see \cite{M-2008},
Chapter 7). The word \textit{uniformly} here means with respect to $(t,x)\in[0,T]\times\R^n$ (see also Section 4 of \cite{M-2014}).
\begin{defin} \label{d:micro_symm}
System \eqref{def:L} is \emph{uniformly symmetrizable} if there exists a $m\times m$ matrix $S(t,x,\xi)$,
 homogeneous of degree $0$ in $\xi$, such that:
 \begin{itemize}
  \item $\xi\,\mapsto\,S(t,x,\xi)$ is $\mc{C}^\infty$ for $\xi\neq0$;
  \item for any point $(t,x,\xi)$, the matrix $S(t,x,\xi)$ is self-adjoint;
  \item there exist constants $0<\lambda\leq\Lambda$ such that $\lambda\,\Id\,\leq\,S(t,x,\xi)\,\leq\,\Lambda\,\Id$ for any $(t,x,\xi)$;
  \item for any point $(t,x,\xi)$, the matrix $S(t,x,\xi)\,\mc{A}(t,x,\xi)$ is self-adjoint.
 \end{itemize}
The matrix valued function $S$ is called a \emph{microlocal symmetrizer} for system \eqref{intro_eq:L}.
\end{defin}

\begin{ex} \label{ex:symm}
Obviously, symmetric systems are microlocally symmetrizable hyperbolic systems, whose symmetrizer is simply the identity matrix.
\end{ex}

\begin{ex} \label{ex:const}
Also hyperbolic systems with constant multiplicities  (and in particular strictly hyperbolic systems, for which multiplicities are all equal to $1$) are microlocally symmetrizable.
Indeed, a symmetrizer can be easily constructed (see e.g. \cite{M-2008}) in terms of the eigenvalues and projection operators onto the eigenspaces related to $\mc{A}(t,x,\xi)$.

In addition, standard perturbation theory for linear operators (see e.g. \cite{K}, \cite{M-2014},\cite{Rauch}) entails that
the eigenvalues and eigenprojectors inherit the same regularity in $(t,x)$ as the coefficients of $L$.
In particular, this implies that, in general, one cannot expect to find a symmetrizer having more smoothness than the one of the
coefficients of $L$ or $\wtilde{L}$.
\end{ex}

In what follows, we are going to consider the case when also the symmetrizer $S$ is log-Lipschitz continuous in $(t,x)$,
in the sense that it verifies an inequality of the same type as \eqref{hyp:LL} at any $\xi\neq0$ fixed. Such a regularity hypothesis for $S$
will be exploited in a fundamental way in order to get our result.
Without loss of generality, we can assume that the constants $K_0$ in \eqref{hyp:bound} and $K_1$ in \eqref{hyp:LL}
are large enough, to control also the corresponding quantities computed for the symmetrizer $S$.

Under the previous hypotheses, we can show an energy estimate with finite loss of derivatives for $L$ and $\wtilde{L}$.
We point out that such a loss is linearly increasing in time, which implies that the solution becomes more and more irregular in the
time evolution.

The estimates are stated for smooth enoug $u$ at this level. However, as a consequence of a ``weak $=$ strong'' type result (see Theorem \ref{th:w-s} below),
they remain true for tempered distributions in a much broader class.
\begin{thm} \label{th:en_LL}
Let us consider the first-order system \eqref{def:L}, and assume it to be microlocally symmetrizable, in the sense of Definition
\ref{d:micro_symm}. Suppose moreover that the coefficients $\bigl(A_j\bigr)_{1\leq j\leq n}$ and the symmetrizer $S$ satisfy
the boundedness and log-Lipschitz conditions \eqref{hyp:bound}-\eqref{hyp:LL}. Suppose also that the coefficient
$B$ verifies hypotheses \eqref{hyp:bound}-\eqref{hyp:Holder}, for some $\g\in\,]0,1[\,$.

Then, for all $s\in\,]0,\g[\,$, there exist positive constants $C_1$, $C_2$ (depending just on $s$, $K_0$ and $K_1$), a $\beta>0$
(depending just on  $K_1$) and a time $T_*\in\,]0,T]$, with $\beta\,T_*\,<\,s$, such that the estimate
\begin{equation} \label{est:u_LL}
\sup_{t\in[0,T_*]}\|u(t)\|_{H^{s-\beta t}}\,\leq\,C_1\,e^{C_2\,T}\,\left(\|u(0)\|_{H^s}\,+\,\int^{T_*}_0
\bigl\|Lu(\tau)\bigr\|_{H^{s-\beta\tau}}\,d\tau\right)
\end{equation}
holds true for any tempered distribution $u\,\in\,L^2\bigl([0,T];H^1(\R^n;\R^m)\bigr)\,\cap\,H^1\bigl([0,T];L^2(\R^n;\R^m)\bigr)$.

An analogous estimate holds true also for operator $\wtilde{L}$ defined in \eqref{def:L*}, but for any $s\,\in\,]-\g,0[\,$ and under the condition $\beta\,T_*\,<\,\g+s$.
\end{thm}

Some remarks on the previous statement are in order.

\begin{rem} \label{r:no-lower}
\begin{itemize}
 \item[(i)] The technical limitation $|s|<\g$ is dictated by product continuity properties (see Proposition
\ref{p:Hol-Sob} below). The same can be said about the
conditions on the time $T_*$.
\item[(ii)] A careful but easy inspection of our proof reveals that, if $B\in L^\infty\bigl([0,T];LL(\R^n)\bigr)$,
then Theorem \ref{th:en_LL} holds true replacing $\g$ by $1$ (see also Corollary \ref{c:LL-H^s}).
On the other hand, having additional regularity for $B$ does not help to improve the result: in particular, this is the case when the
operator is homogeneous of first order, i.e. if $B\equiv0$.
\item[(iii)] In the case of operator $\wtilde{L}$, the H\"older regularity of the $0$-th order term imposes an additional limitation on
the lifespan $T_*$. We notice that this is coherent with what is known for scalar wave equations (see e.g. \cite{C-DS-F-M_Birk}).
\end{itemize}
\end{rem}

From the previous theorem, we can deduce the existence and uniqueness of a local in time solution to the Cauchy problem
associated to $L$ and $\wtilde{L}$.

\begin{thm} \label{t:global_e}
Let us consider the first-order system \eqref{def:L}, and assume it to be microlocally symmetrizable, in the sense of Definition
\ref{d:micro_symm}. Suppose moreover that the coefficients $\bigl(A_j\bigr)_{1\leq j\leq n}$ and the symmetrizer $S$ satisfy
the boundedness and log-Lipschitz conditions \eqref{hyp:bound}-\eqref{hyp:LL}. Suppose also that the coefficient
$B$ verifies hypotheses \eqref{hyp:bound}-\eqref{hyp:Holder}, for some $\g\in\,]0,1[\,$.

For $s\in\,]0,\g[\,$, let $\beta>0$ and $T_*>0$ respectively the loss parameter and the existence time given by Theorem \ref{th:en_LL}.
Set $s_0\,:=\,s-\beta T_*$.

Then, for any fixed $u_0\in H^s(\R^n;\R^m)$ and $f\in L^1\bigl([0,T];H^s(\R^n;\R^m)\bigr)$, there exists a unique solution $u\,\in\,\mc{C}\bigl([0,T_*];H^{s_0}(\R^n;\R^m)\bigr)$
to the Cauchy problem
\begin{equation} \label{eq:Cauchy}
\left\{\begin{array}{l}
        Lu\;=\;f \\[1ex]
	u_{|t=0}\;=\;u_0\,,
       \end{array}\right.
\end{equation}
which satisfies the energy inequality \eqref{est:u_LL}. In particular, for any $t\in[0,T_*]$, one has
$$
u(t)\,\in\,H^{s-\beta t}(\R^n;\R^m)\,,
$$
and the map $t\,\mapsto\,u(t)$ is continuous between the respective functional spaces.

An analogous statement still holds true for the conservative operator $\wtilde{L}$.
\end{thm}

\begin{rem} \label{r:Cauchy}
Let us denote by $\mc{C}_b^\infty(\R^n)$ the space of $\mc{C}^\infty(\R^n)$ functions which are uniformly bounded with all their derivatives.

Theorems \ref{th:en_LL} and \ref{t:global_e} immediately imply the following statement: if the coefficients of $L$ (respectively $\wtilde{L}$) and the symmetrizer $S$ are
$L^\infty\bigl([0,T];\mc{C}^\infty_b(\R^n)\bigr)$, with the first order coefficients and $S$ log-Lipschitz continuous in time, then the Cauchy
problem for $L$ (respectively $\wtilde{L}$) is well-posed in the space $H^\infty(\R^n;\R^m)$, with a finite loss of derivatives.
\end{rem}

\subsection{Local in space results} \label{ss:local_th}
We turn now our attention to the local in space problem. For simplicity of exposition, we will always consider \emph{smooth} bounded
domains and manifolds.

Hence, we fix a smooth open bounded domain $\Omega\subset\R^{1+n}$, and
we suppose that the coefficients $A_j=A_j(z)$ are log-Lipschitz in $\Omega$ (keep in mind Definition \ref{d:LL}).

We recall also that $H^s(\Omega)$ is defined (see Chapter 3 of \cite{T_1983}, where the more general context of Besov spaces is treated) as
\begin{equation} \label{def:H^s}
H^s(\Omega)\,:=\,\biggl\{u\in\mc{D}'(\Omega)\;\biggl|\;\mbox{ there exists }\;\wtilde{u}\,\in\,H^s(\R^{1+n})
\;\mbox{ such that }\;\wtilde{u}_{|\Omega}\,\equiv\,u\biggr\}\,,
\end{equation}
endowed with the norm
$$
\|u\|_{H^s(\Omega)}\,:=\,\inf\,\biggl\{\left\|\wtilde{u}\right\|_{H^s(\R^{1+n})}\;\biggl|\quad\wtilde{u}\,\in\,H^s(\R^{1+n})\quad
\mbox{ and }\quad\wtilde{u}_{|\Omega}\,\equiv\,u\biggr\}\,.
$$
As for the boundary $\d\Omega$, partition of unity leads to a similar definition for $H^s(\d\Omega)$, by use of local charts
and extension operator; the same can be said for smooth submanifolds $\varSigma$.

Let us consider the operator $P$, defined in $\Omega$ by the formula
\begin{equation} \label{eq:def_P}
P(z,\d_z)u\,:=\,\sum_{j=0}^nA_j(z)\d_{z_j}u\,+\,B(z)u\,,
\end{equation}
where, for all $1\leq j\leq n$, the $A_j$'s and $B$ are real $m\times m$ matrices; we will specify later on their regularity. For the time being,
let us introduce the principal symbol $P_1$ of $P$, identified by the formula
$$
P_1(z,\z)\,=\,\sum_{j=0}^ni\,\z^j\,A_j(z)\,,
$$
and recall some basic definitions (see Section 4 of \cite{M-2014}). At this level, the dependence of the coefficients on the variable $z\in\Omega$ is not really important,
so let us omit it from the notations here.
\begin{defin} \label{d:hyperbolic}
The operator $P$ is said to be \emph{hyperbolic} in the direction $\nu\in\R^{1+d}$ if the following conditions are verified:
\begin{itemize}
 \item[(i)] $\det\bigl(P_1(\nu)\bigr)\,\neq\,0$;
 \item[(ii)] there exists a $\eta_0>0$ such that $\det\bigl(P_1(i\tau\nu+\z)+B\bigr)\,\neq\,0$ for all $\z\in\R^{1+d}$ and all $\tau\in\R$ such that $|\tau|>\eta_0$.
\end{itemize}
The principal operator $P_1(\d_z)$ is \emph{strongly hyperbolic} in the direction $\nu$ if for all matrix $B\in\mc{M}_m(\R)$, then $P_1+B$ is hyperbolic in the direction $\nu$.
\end{defin}

Let us recall that Proposition 4.2 of \cite{M-2014} gives a characterization of the strong hyperbolicity. We do not enter into the details here; however, we will come
back to this notion in a while.

\medbreak
Now, we are interested in considering the dependence on $z\in\Omega$: more precisely, for any $z\in\Omega$, we assign a direction $\nu(z)\in\mbb{S}^d$
(where $\mbb{S}^d$ is the unitary sphere in $\R^{1+d}$) in a smooth way.
In what follows, we are going to assume that
\begin{itemize}
 \item[\bf (H-1)] $P_1(z,\cdot)$ is \emph{uniformly strongly hyperbolic} in the direction $\nu(z)$, for all $z\in\Omega$.
\end{itemize}
By Theorem 4.10 of \cite{M-2014}, hypothesis \textbf{(H-1)} is equivalent to the following conditions:
\begin{enumerate}
 \item one has the property
\begin{equation} \label{eq:det_pos}
C_\Omega\,:=\,\inf_{z\in\Omega}\,\Bigl|{\rm det}\,P_1\bigl(z,\nu(z)\bigr)\Bigr|\,>\,0\,;
\end{equation}
 \item $P_1$ admits a bounded family of \emph{full symmetrizers}
$\mbf{S}(z,\cdot)$ which is \emph{uniformly positive} in the direction $\nu(z)$.
\end{enumerate}
Hence, let us recall Definition 4.7 of \cite{M-2014} here below.
\begin{defin} \label{d:full-symm}
A bounded family of \emph{full symmetrizers} for $P_1(z,\z)$ is a family of $m\times m$ matrices $\mbf{S}(z,\z)$, homogeneous of degree $0$ in
$\z\neq 0$, such that the following conditions are satisfied:
\begin{itemize}
\item uniform  boundedness: there exists a constant $\Lambda>0$ such that $\sup_{(z,\z)}|\mbf{S}(z,\z)|_{\mc{M}}\,\leq\,\Lambda$;
\item symmetrizability: for any $(z,\z)$, the matrix $\mbf{S}(z,\z)\,P_1(z,\z)$ is self-adjoint.
 \end{itemize}
The symmetrizer $\mbf{S}(z,\z)$ is \emph{positive} in the direction $\nu\neq0$ if there exists a constant $\lambda>0$
such that, for all $\z\neq0$, one has
$$
v\,\in\,{\rm Ker}\,P_1(z,\zeta)\qquad\Longrightarrow\qquad
\Re\Bigl(\mbf{S}(z,\z)\,P_1\bigl(z,\nu\bigr)v\,\cdot\,v\Bigr)\,\geq\,\lambda\,|v|^2\,.
$$
\emph{Uniform positivity} means positivity of $\mbf{S}(z,\cdot)$ in the direction $\nu(z)$ for all $z\in\Omega$, for a constant $\lambda$ independent of $z$.
\end{defin}

We are going to need also the following assumptions concerning the family of symmetrizers:
\begin{itemize}
 \item[\bf (H-2)] the map $\z\,\mapsto\,\mbf{S}(z,\z)$ is $\mc{C}^\infty$ for $\z\neq0$ (smoothness in $\z$);
\item[\bf (H-3)] the map $z\,\mapsto\,\mbf{S}(z,\z)$ is uniformly $LL$ in $\Omega$ (log-Lipschitz regularity in $z$).
\end{itemize}
As for the coefficients of the operator $P$, defined in \eqref{eq:def_P}, we suppose instead that:
\begin{itemize}
 \item[\bf (H-4)] the matrices $A_j$ have coefficients in the $LL(\Omega)$ class;
\item[\bf (H-5)] $B$ has coefficients in the  H\"older space $\mc{C}^\g(\Omega)$.
\end{itemize}

Let us now fix a smooth hypersurface $\varSigma\subset\Omega$.
As in \cite{C-M}, up to  shrink $\Omega$, we can assume that $\varSigma$ is defined by the equation $\vphi=0$, for a smooth $\vphi$
such that $d\vphi\neq0$. Finally, we suppose that, for any $z\in\varSigma$, the vector $\nu(z)$ coincides
with the normal to $\varSigma$ in $z$, i.e. $d\vphi(z)$.

We introduce the notations $\Omega_\geq\,:=\,\Omega\,\cap\,\{\vphi\geq0\}$ and
$\Omega_>\,:=\,\Omega\,\cap\,\{\vphi>0\}$.
For $s\in\R$, we say that $u\in H^s_{loc}(\Omega_\geq)$ if, for any open $\Omega'\subset\Omega$, relatively compact in $\Omega$, the
restriction of $u$ to $\Omega'\,\cap\,\{\vphi>0\}$ belongs to $H^s(\Omega_>)$. In a similar way, we say that $v\in H^s_{comp}(\Omega_\geq)$
if $v\in H^s(\Omega_>)$ has compact support in $\Omega_\geq$.

For a $z_0\in\varSigma$, we are interested in solving the Cauchy problem for $P$ in a neighborhood of $z_0$.
We have the following \emph{local existence} result.

\begin{thm} \label{t:local_e}
Let $1/2<\g<1$ and $s\in\,]1-\g,\g[\,$. Let $P$ be the operator defined in \eqref{eq:def_P}, satisfying all hypotheses from \emph{\textbf{(H-1)}} to \emph{\textbf{(H-5)}}.

Fix a neighborhood $\omega$ of $z_0$ in $\varSigma$. Then, there exist a $s_0\in\,]1-\g,s[\,$ and a neighborhood $\Omega_0$ of $z_0$ in $\Omega$
such that, for any $u_0\in H^s(\omega)$ and any $f\in H^s(\Omega_0\cap\{\vphi>0\})$, there exists a solution
$u\in H^{s_0}(\Omega_0\cap\{\vphi>0\})$ to the Cauchy problem
$$
\left\{\begin{array}{l}
        Pu\;=\;f \\[1ex]
        u_{|\varSigma}\;=\;u_0\,.
       \end{array}\right. \leqno{(C\!P)}
$$
\end{thm}

\begin{rem} \label{r:sense}
Notice that, due to the low regularity of the coefficients, the meaning of $(C\!P)$ is not clear \textsl{a priori}: we will precise it in Subsection \ref{ss:sense}.
In particular, one step of the proof is devoted exactly to giving sense to the trace operator in a weak smoothness framework.
\end{rem}

For the problem $(C\!P)$ we have also \emph{local uniqueness} of a solution. From this statement, which establish propagation of zero across any space-like hypersurface, one can deduce
also further local results, for instance about finite propagation speed and domain of dependence, by use of classical arguments (see e.g. \cite{J-M-R_2005} and \cite{Rauch_2005}). 
\begin{thm} \label{t:local_u}
Let $1/2<\g<1$ and $s\in\,]1-\g,\g[\,$. Let $P$ be the operator defined in \eqref{eq:def_P}, satisfying all hypotheses from \emph{\textbf{(H-1)}} to \emph{\textbf{(H-5)}}.

If $u\in H^{s}(\Omega_>)$ satisfies $(C\!P)$ with $f=0$ and Cauchy datum
$u_0=0$, then $u\equiv0$ on a neighborhood of $z_0$ in $\Omega_\geq$.
\end{thm}

In Section \ref{s:local}, after rigorously justify the good formulation of the Cauchy problem $(C\!P)$, we will show the invariance
of our hypotheses with respect to smooth changes of variables. This fact will enable us to pass in $(t,x)$ coordinates: then, the proof
of Theorems \ref{t:local_e} and \ref{t:local_u} will be deduced from the global in space results. Therefore, let us focus first
on these latter properties.

\section{Tools from Littlewood-Paley theory} \label{s:tools}

We collect here some notions and results which turn out to be useful in our proof. 
First, by use of the Littlewood-Paley decomposition, we describe some properties of Sobolev spaces and of log-Lipschitz
functions. Then, we recall some notions of Paradifferential Calculus, focusing on operators whose symbol is in the log-Lipschitz class.

\subsection{Dyadic analysis of Sobolev and log-Lipschitz classes} \label{ss:L-P}

Let us first define the so called \emph{Littlewood-Paley decomposition} in $\R^N$ (for any $N\geq1$), based on a non-homogeneous dyadic
partition of unity with respect to the Fourier variable. We refer to \cite{B-C-D} (Chapter 2) and \cite{M-2008} (Chapters 4 and 5)
for the details.

So, fix a smooth radial function $\chi$ supported in the ball $B(0,2)\subset\R^N$, equal to $1$ in a neighborhood of $B(0,1)$
and such that $r\mapsto\chi(r\,e)$ is nonincreasing over $\R_+$ for all unitary vectors $e\in\R^N$. Set
$\varphi\left(\xi\right)=\chi\left(\xi\right)-\chi\left(2\xi\right)$ and $\vphi_j(\xi):=\vphi(2^{-j}\xi)$ for all $j\geq0$.

The dyadic blocks $(\Delta_j)_{j\in\Z}$ are defined by\footnote{Throughout we agree  that  $f(D)$ stands for 
the pseudo-differential operator $u\mapsto\mc{F}^{-1}\bigl(f(\xi)\,\mc{F}u(\xi)\bigr)$.}
$$
\Delta_j:=0\ \hbox{ if }\ j\leq-1,\quad\Delta_{0}:=\chi(D)\quad\hbox{and}\quad
\Delta_j:=\varphi(2^{-j}D)\ \text{ if }\  j\geq1.
$$
We  also introduce the following low frequency cut-off:
$$
S_ju\,:=\,\chi(2^{-j}D)\,u\,=\,\sum_{k\leq j}\Delta_{k}u\quad\text{for}\quad j\geq0.
$$

By use of the previous operators, for any $u\in\mc{S}'$, we have the Littlewood-Paley decomposition of $u$: namely,
the equality $u=\sum_{j}\Delta_ju$ holds true in $\mc{S}'$.

Let us recall the fundamental \emph{Bernstein's inequalities}.
  \begin{lemma} \label{l:bern}
Let  $0<r<R$.   A constant $C$ exists so that, for any nonnegative integer $k$, any couple $(p,q)$ 
in $[1,+\infty]^2$ with  $p\leq q$  and any function $u\in L^p$,  we  have, for all $\lambda>0$,
$$
\displaylines{
{\rm supp}\, \widehat u \subset   B(0,\lambda R)\quad
\Longrightarrow\quad
\|\nabla^k u\|_{L^q}\, \leq\,
 C^{k+1}\,\lambda^{k+N\left(\frac{1}{p}-\frac{1}{q}\right)}\,\|u\|_{L^p}\;;\cr
{\rm supp}\, \widehat u \subset \{\xi\in\R^N\,|\, r\lambda\leq|\xi|\leq R\lambda\}
\quad\Longrightarrow\quad C^{-k-1}\,\lambda^k\|u\|_{L^p}\,
\leq\,
\|\nabla^k u\|_{L^p}\,
\leq\,
C^{k+1} \, \lambda^k\|u\|_{L^p}\,.
}$$
\end{lemma}   


Let us now introduce the class of \emph{logarithmic Sobolev spaces}, which naturally come into play in the study
of hyperbolic operators with low regularity coefficients (see \cite{C-M}, \cite{C-DS-F-M_tl} and \cite{C-DS-F-M_wp}).
Let us set $\Pi(D)\,:=\,\log(2+|D|)$, i.e. its symbol is $\pi(\xi)\,:=\,\log(2+|\xi|)$.
\begin{defin} \label{d:log-H^s}
 For all $(s,\alpha)\in\R^2$, we define the space $H^{s+\alpha\log}(\R^N)$ as $\Pi^{-\alpha}H^s(\R^N)$, i.e.
$$
f\,\in\,H^{s+\alpha\log}\quad\Longleftrightarrow\quad\Pi^\alpha f\,\in\,H^s\quad\Longleftrightarrow\quad
\pi^\alpha(\xi)\left(1+|\xi|^2\right)^{s/2}\what{f}(\xi)\,\in\,L^2\,.
$$
\end{defin}
Obviously, for $\alpha=0$ one recovers the classical Sobolev space $H^s$.

We have the following dyadic characterization of these classes (see \cite{M-2008}, Proposition 4.1.11),
which generalizes the classical property for the $H^s$ scale.
\begin{prop} \label{p:log-H}
 Let $s$, $\alpha\,\in\R$. Then $u\in\mc{S}'$ belongs to the space $H^{s+\alpha\log}$ if and only if:
\begin{itemize}
 \item[(i)] for all $k\in\N$, $\Delta_ku\in L^2(\R^N)$;
\item[(ii)] set $\,\delta_k\,:=\,2^{ks}\,(1+k)^\alpha\,\|\Delta_ku\|_{L^2}$ for all $k\in\N$, the sequence
$\left(\delta_k\right)_k$ belongs to $\ell^2(\N)$.
\end{itemize}
Moreover, $\|u\|_{H^{s+\alpha\log}}\,\sim\,\left\|\left(\delta_k\right)_k\right\|_{\ell^2}$.
\end{prop}

The previous proposition can be summarized by the equivalence $H^{s+\alpha\log}\,\equiv\,B^{s+\alpha\log}_{2,2}$, where,
for any $(s,\alpha)\in\R^2$ and $1\leq p,r\leq+\infty$, the \emph{non-homogeneous logarithmic Besov space}
$B^{s+\alpha\log}_{p,r}$ is the subset of tempered distributions $u$ for which
\begin{equation} \label{eq:log-Besov}
\|u\|_{B^{s+\alpha\log}_{p,r}}\,:=\,
\left\|\left(2^{js}\,(1+j)^{\alpha}\,\|\Delta_ju\|_{L^p}\right)_{j\in\N}\right\|_{\ell^r}\,<\,+\infty\,.
\end{equation}
In addition, we point out that an analogous characterization holds true also for H\"older classes: namely, for any $\gamma\in\,]0,1[\,$
one has $\mc{C}^\gamma\,\equiv\,B^{\gamma}_{\infty,\infty}\,=\,B^{\gamma+0\log}_{\infty,\infty}$.

We recall that the previous definition and properties do not depend on the choice of the cut-off functions used in a Littlewood-Paley
decomposition: in the logarithmic framework, this comes from Lemma 3.5 of \cite{C-DS-F-M_Z-sys}, which we recall here.
\begin{lemma} \label{l:log-B_ind}
 Let $\mc{C}\subset\R^d$ be a ring, $(s,\alpha)\in\R^2$ and $(p,r)\in[1,+\infty]^2$. Let $\left(u_j\right)_{j\in\N}$ be
a sequence of smooth functions such that
$$
{\rm supp}\,\what{u}_j\,\subset\,2^j\,\mc{C}\qquad\quad\mbox{ and }\qquad\quad
\left\|\left(2^{js}\,(1+j)^\alpha\,\|u_j\|_{L^p}\right)_{j\in\N}\right\|_{\ell^r}\,<\,+\infty\,.
$$

Then $u:=\sum_{j\in\N}u_j$ belongs to $B^{s+\alpha\log}_{p,r}$ and
$\|u\|_{B^{s+\alpha\log}_{p,r}}\,\leq\,C_{s,\alpha}\,\left\|\left(2^{js}\,(1+j)^\alpha\,
\|u_j\|_{L^p}\right)_{j\in\N}\right\|_{\ell^r}$.
\end{lemma}
This fact will be used freely throughout the paper.

Now we mention a couple of results which will be useful in the sequel. They are classical (see \cite{B-C-D},
Chapter 2), and their extension to the logarithmic framework is proved in \cite{C-DS-F-M_Z-sys}.
\begin{lemma} \label{l:log-S_j}
Fix $(s,\alpha)\in\R^2$ and let $u\in\mc{S}'$ given.
\begin{itemize}
\item[(i)] If the sequence $\bigl(2^{js}\,(1+j)^\alpha\,\|S_ju\|_{L^2}\bigr)_{j\in\N}$
belongs to $\ell^2$, then $u\in H^{s+\alpha\log}$ and
$$
\|u\|_{H^{s+\alpha\log}}\;\leq\;C\,\left\|\bigl(2^{js}\,(1+j)^\alpha\,\|S_ju\|_{L^2}\bigr)_{j\in\N}\right\|_{\ell^2}\,,
$$
for some constant $C>0$ depending only on $s$ and $\alpha$, but not on $u$.
\item[(ii)] Suppose $u\in H^{s+\alpha\log}$, with $s<0$. Then the sequence
$\bigl(2^{js}\,(1+j)^\alpha\,\|S_ju\|_{L^2}\bigr)_{j\in\N}\,\in\,\ell^2$, and
$$
\left\|\bigl(2^{js}\,(1+j)^\alpha\,\|S_ju\|_{L^2}\bigr)_{j\in\N}\right\|_{\ell^2}\;\leq\;
\wtilde{C}\,\|u\|_{H^{s+\alpha\log}}\,,
$$
for some constant $\wtilde{C}>0$ depending only on $s$ and $\alpha$.
\end{itemize}.
\end{lemma}

Observe that, in general, the second property fails in the endpoint case $s=0$. Indeed, for $s=0$ one can only infer, for any $\alpha\leq0$,
$$
\left\|\biggl((1+j)^\alpha\,\|S_ju\|_{L^2}\biggr)_{j\in\N}\right\|_{\ell^\infty}\;\leq\;\wtilde{C}\,\|u\|_{B^{0+\alpha\log}_{2,1}}\,.
$$

The second result we want to mention ia a sort of ``dual'' version of the previous lemma.
\begin{lemma} \label{l:log-ball}
 Let $\mc{B}$ be a ball of $\R^N$ and take $s>0$ and $\alpha\in\R$.
Let $\left(u_j\right)_{j\in\N}$ be a sequence of smooth functions such that
$$
{\rm supp}\,\what{u}_j\,\subset\,2^j\mc{B}\qquad\mbox{ and }\qquad
\bigl(2^{js}\,(1+j)^\alpha\,\left\|u_j\right\|_{L^2}\bigr)_{j\in\N}\,\in\,\ell^2\,.
$$

Then the function $\,u\,:=\,\sum_{j\in\N}u_j\,$ belongs to the space $H^{s+\alpha\log}$, and there exists a constant
$C$, depending only on $s$ and $\alpha$, such that
$$
\|u\|_{H^{s+\alpha\log}}\,\leq\,C\,\left\|\left(2^{js}\,(1+j)^\alpha\,\left\|u_j\right\|_{L^2}\right)_{j\in\N}\right\|_{\ell^2}\,.
$$
\end{lemma}

Once again, the previous statement is not true in the endpoint case $s=0$: then, one can just infer, for any $\alpha\geq0$,
$$
\|u\|_{B^{0+\alpha\log}_{2,\infty}}\,\leq\,C\,\left\|\biggl((1+j)^\alpha\,\left\|u_j\right\|_{L^p}\biggr)_{j\in\N}\right\|_{\ell^1}\,.
$$

We now turn our attention to the study of the class of \emph{log-Lipschitz functions}. We have given the general definition
in Definition \ref{d:LL}; now, we restrict to the case $\Omega=\R^N$, for some $N\geq1$.

Let us recall some properties which can be deduced by use of dyadic decomposition
(see \cite{C-L} and \cite{C-M} for the proof), and which are true in any dimension $N\geq1$.
\begin{prop} \label{p:dyadic-LL}
 There exists a positive constant $C$ such that, for all $a\in LL(\R^N)$ and all integers $k\geq0$, we have
\begin{eqnarray*}
 \left\|\Delta_k a\right\|_{L^\infty} & \leq & C\,(k+1)\,2^{-k}\,\|a\|_{LL} \\ 
 \left\|a\,-\,S_k a\right\|_{L^\infty} & \leq & C\,(k+1)\,2^{-k}\,\|a\|_{LL} \\ 
\left\|S_k a\right\|_{\rm Lip}\,:=\,\left\|S_k a\right\|_{L^\infty}\,+\,\left\|\nabla S_k a\right\|_{L^\infty} & \leq &
C\,(k+1)\,\|a\|_{LL}\,. 
\end{eqnarray*}
\end{prop}

\begin{rem} \label{r:LL_char}
By Proposition 3.3 of \cite{C-L}, the last property is a characterization of the space $LL$.
\end{rem}

We conclude this part by showing continuity propertis of multiplication of Sobolev distributions by H\"older-type functions.
\begin{prop} \label{p:Hol-Sob}
Let $b\in B^{\g+\vrho\log}_{\infty,\infty}$, where $\g>0$ and $\vrho\in\R$, or $\g=0$ and $\vrho>1$. Then the multiplication operator
$\;u\,\mapsto\,b\,u\;$ is a continuous self-map of $H^{s+\alpha\log}(\R^N)$ if:
\begin{itemize}
 \item $|s|<\g$, no matter the value of $\alpha\in\R$;
 \item $s=\g$ and $\alpha<\vrho-1/2$, or $s=-\g$ and $\alpha>1/2-\vrho$.
\end{itemize}
\end{prop}

\begin{proof}
We use Bony's paraproduct decomposition (see \cite{Bony}, \cite{B-C-D} and \cite{M-2008}) to write
$$
b\,u\,=\,T_bu\,+\,T_ub\,+\,R(b,u)\,=\,T_bu\,+\,T'_ub\,,
$$
where the previous operators are defined by the formulas
\begin{equation} \label{eq:paraprod}
T_bu\,=\,\sum_jS_{j-3}b\,\Delta_ju\,,\qquad R(b,u)\,=\,\sum_j\sum_{|j-k|\leq3}\Delta_jb\,\Delta_ku\,\qquad
T'_ua\,=\,\sum_jS_{j+3}u\,\Delta_jb\,.
\end{equation}

We remark that the conditions on $\g$ and $\vrho$ imply the chain of embeddings
$B^{\g+\vrho\log}_{\infty,\infty}\,\hookrightarrow\,B^{0}_{\infty,1}\,\hookrightarrow\,L^\infty$.
Then, by classical properties of paraproduct, we immediately have that $T_bu\,\in\,H^{s+\alpha\log}$, with the estimate
$\|T_bu\|_{H^{s+\alpha\log}}\,\leq\,C\,\|b\|_{L^\infty}\,\|u\|_{H^{s+\alpha\log}}$.

Let us now focus on the case $s>0$ (and then $\g>0$), and let us consider the operator $T'_ub$. For any $j\in\N$,
using Bernstein inequalities and definition \eqref{eq:log-Besov}, we deduce
$$ 
\left\|S_{j+3}u\,\Delta_jb\right\|_{L^2}\,\leq\,\left\|S_{j+3}u\right\|_{L^2}\,\left\|\Delta_jb\right\|_{L^\infty}\,\leq\,
C\,2^{-j\g}\,(j+1)^{-\vrho}\,\|u\|_{H^{s+\alpha\log}}\,\|b\|_{B^{\g+\vrho\log}_{\infty,\infty}}\,.
$$ 
Since the term $S_{j+3}u\,\Delta_jb$ is supported in dyadic balls $2^j\mc{B}$, we conclude by use of Lemma \ref{l:log-ball}.
We point out that, for $s=\g$, the condition $\alpha<\vrho-1/2$ is needed to have the right-hand side of the previous
inequality in $\ell^2$.

For $s\leq0$ (and then $\g$ can be taken even $0$), instead, we employ the finer decomposition in $T_ub+R(b,u)$.
First of all, by Lemma \ref{l:log-S_j}  and \eqref{eq:log-Besov} again, we have the estimate
$$
\left\|S_{j-3}u\,\Delta_jb\right\|_{L^2}\,\leq\,\left\|S_{j-3}u\right\|_{L^2}\,\left\|\Delta_jb\right\|_{L^\infty}\,\leq\,
C\,\|u\|_{H^{s+\alpha\log}}\,\|b\|_{B^{\g+\delta\log}_{\infty,\infty}}\,2^{-j(\g+s)}\,(1+j)^{-(\alpha+\vrho)}\,\zeta_j\,,
$$
for some $\bigl(\zeta_j\bigr)_j\,\in\,\ell^2$ of unitary norm. Observe that, for $s=0$, we have no more the presence
of $\bigl(\zeta_j\bigr)_j$, but the right-hand side still belongs to $\ell^2$, thanks to our hypotheses on $\g$, $\alpha$ and
$\delta$.
Since the generic term $S_{j-3}u\,\Delta_jb$ is supported in dyadic
rings $2^j\mc{C}$, from Lemma \ref{l:log-B_ind} we infer
$\|T_ub\|_{H^{s+\alpha\log}}\,\leq\,C\,\|u\|_{H^{s+\alpha\log}}\,\|b\|_{B^{\g+\vrho\log}_{\infty,\infty}}$.

For the remainder term, we use again Lemma \ref{l:log-B_ind}: focusing just on the ``diagonal'' term 
$R_0(b,u)\,=\,\sum_j\Delta_jb\,\Delta_ju$ (the other ones being similar),
we have to bound, for any $\nu\geq0$, the quantity $2^{s\nu}\,(1+\nu)^\alpha\,\left\|\Delta_\nu R_0(b,u)\right\|_{L^2}$.
By use of Proposition \ref{p:log-H} we have
$$ 
\left\|\Delta_\nu R_0(b,u)\right\|_{L^2}\,\leq\,\sum_{j\geq\nu-3}\left\|\Delta_jb\,\Delta_ju\right\|_{L^2}\,\leq\,
C\,\|u\|_{H^{s+\alpha\log}}\,\|b\|_{B^{\g+\vrho\log}_{\infty,\infty}}\,\sum_{j\geq\nu-3}2^{-(\g+s)j}\,(1+j)^{-(\alpha+\vrho)}\,\zeta_j\,,
$$ 
where the sequence $\bigl(\zeta_j\bigr)_j$ is as above. Then, by Cauchy-Schwarz inequality we get
\begin{eqnarray*}
2^{s\nu}\,(1+\nu)^\alpha\,\left\|\Delta_\nu R_0(b,u)\right\|_{L^2} & \leq & C\,\|u\|_{H^{s+\alpha\log}}\,
\|b\|_{B^{\g+\vrho\log}_{\infty,\infty}}\,\times \\
& & \qquad \times\,2^{s\nu}\,(1+\nu)^\alpha\,\left(\sum_{j\geq\nu-3}2^{-2(\g+s)j}\,
(1+j)^{-2(\vrho+\alpha)}\right)^{\!\!1/2}\,,
\end{eqnarray*}
and this completes the proof of the case $s\leq0$, and therefore of the proposition.
\end{proof}

We point out that, by Proposition \ref{p:dyadic-LL}, one has $LL\,\hookrightarrow\,B^{1-\log}_{\infty,\infty}$. Hence,
from the previous statement, we immediately infer continuity properties of multiplication by log-Lipschitz functions,
which generalize Proposition 3.5 of \cite{C-L} to the framework of logarithmic Sobolev spaces.
\begin{coroll} \label{c:LL-H^s}
Let  $a\in LL(\R^N)$. Then the multiplication operator $\;u\,\mapsto\,a\,u\;$ is a continuous map of $H^{s+\alpha\log}(\R^N)$
into itself if:
\begin{itemize}
 \item $|s|<1$, no matter the value of $\alpha\in\R$;
 \item $s=1$ and $\alpha<-3/2$, or $s=-1$ and $\alpha>3/2$.
\end{itemize}
\end{coroll}

\subsection{Paradifferential Calculus in the log-Lipschitz class} \label{ss:paradiff}

In our study, we need to resort to tools from Paradifferential Calculus, as introduced by J.-M. Bony in the celebrated paper \cite{Bony}.
We refer to \cite{B-C-D} (Chapter 2) for a complete treatement, and to papers \cite{M-1986}-\cite{M-Z} for a construction depending
on parameters. Here, we follow the approach of \cite{M-2008} (see Chapters 4 and 5).

The first part of this section is devoted to recall basic properties: we adapt the classical construction to consider symbols having
log-Lipschitz smoothness with respect to the space variable, and we define general paradifferential operators associated to them,
for which we develop also a symbolic calculus. In the final part, we consider the case of time-dependent symbols,
which are log-Lipschitz in $t$: at this point, time cannot be considered as a parameter anymore, and we need to establish some properties
for paradifferential operators whose symbols belong to this class.


\subsubsection{Symbols having log-Lipschitz regularity} \label{sss:LL}

Fix a cut-off function $\psi\in\mc{C}^\infty(\R^N\times\R^N)$ which verifies the following properties:
\begin{itemize}
 \item there exist $0<\veps_1<\veps_2<1$ such that
$$
\psi(\eta,\xi)\,=\,\left\{\begin{array}{lcl}
                           1 & \mbox{for} & |\eta|\leq\veps_1\left(1+|\xi|\right) \\ [1ex]
			   0 & \mbox{for} & |\eta|\geq\veps_2\left(1+|\xi|\right)\,;
                          \end{array}
\right.
$$
\item for all $(\beta,\alpha)\in\N^N\times\N^N$, there exists a constant $C_{\beta,\alpha}>0$ such that
$$
\left|\d^\beta_\eta\d^\alpha_\xi\psi(\eta,\xi)\right|\,\leq\,C_{\beta,\alpha}\left(1+|\xi|\right)^{-|\alpha|-|\beta|}\,.
$$
\end{itemize}

For instance, it is easy to verify (see Ex. 5.1.5 \cite{M-2008}) that
$$
\psi(\eta,\xi)\,\equiv\,\psi_{-3}(\eta,\xi)\,:=\,\sum_{k=0}^{+\infty}\chi_{k-3}(\eta)\,\vphi_k(\xi)\,,
$$
where $\chi$ and $\vphi$ are the localization (in phase space) functions associated to a Littlewood-Paley decomposition,
satisfies the previous requirements.

Define now $G^\psi$ as the inverse Fourier transform of $\psi$ with respect to the variable $\eta$:
$$
G^\psi(x,\xi)\,:=\,\left(\mc{F}^{-1}_\eta\psi\right)(x,\xi)\,.
$$
We have the following result (see Lemma 5.1.7 of \cite{M-2008}).

\begin{lemma} \label{l:G}
 For all $(\beta,\alpha)\in\N^N\times\N^N$, there exist constants $C_{\beta,\alpha}>0$ such that:
\begin{eqnarray*}
\left\|\d^\beta_x\d^\alpha_\xi G^\psi(\cdot,\xi)\right\|_{L^1(\R^N_x)} & \leq & 
C_{\beta,\alpha}\left(1+|\xi|\right)^{-|\alpha|+|\beta|} \\ 
\left\||\cdot|\,\log\left(1+\frac{1}{|\cdot|}\right)\,\d^\beta_x\d^\alpha_\xi G^\psi(\cdot,\xi)\right\|_{L^1(\R^N_x)}
& \leq & C_{\beta,\alpha}\left(1+|\xi|\right)^{-|\alpha|+|\beta|-1}\,\log(2+|\xi|)\,.
\end{eqnarray*}
\end{lemma}

Let us now take a symbol $a=a(x,\xi)$: thanks to $G^\psi$, we can smooth it out in the space variable, and then define the paradifferential
operator associated to $a$ as the pseudodifferential operator related to this smooth function.

First of all, let us specify the class of symbols we are interested in.
\begin{defin} \label{d:symbols}
Let $X\subset L^\infty(\R^N)$ a Banach space and fix $(m,\delta)\in\R^2$.
\begin{itemize}
\item[(i)] We denote by $\Gamma^{m+\delta\log}_{X}$ the space of functions $a(x,\xi)$ which are locally
bounded over $\R^N\times\R^N$, of class $\mc{C}^\infty$ with respect to $\xi$ and which satisfy
the following property: for all $\alpha\in\N^N$ and all $\xi\in\R^N$, the map $x\,\mapsto\,\d^\alpha_\xi a(x,\xi)$ belongs to
$X$, and, for some $C_\alpha>0$,
$$
\left\|\d^\alpha_\xi a(\,\cdot\,,\xi)\right\|_{X}\;\leq\;C_\alpha\,(1+|\xi|)^{m-|\alpha|}\,\log^\delta(2+|\xi|)\,.
$$
\item[(ii)] $\Sigma^{m+\delta\log}_{X}$ is the space of symbols $\sigma\in\Gamma^{m+\delta\log}_{X}$ which
satisfy the following spectral condition: there exists a $0<\epsilon<1$ such that, for all $\xi\in\R^N$, the spectrum of the function
$x\,\mapsto\,\sigma(x,\xi)$ is contained in the ball $\left\{|\eta|\,\leq\,\epsilon\,(1+|\xi|)\right\}$.
\end{itemize}
\end{defin}

In a quite natural way, we can equip $\Gamma^{m+\delta\log}_X$ with the family of seminorms
\begin{equation} \label{eq:L-inf_sem}
\|a\|^{(m,\delta)}_{(X,k)}\;:=\;\sup_{|\alpha|\leq k}\,\sup_{\R^N_\xi}
\left((1+|\xi|)^{-m+|\alpha|}\,\log^{-\delta}(2+|\xi|)\,\left\|\d^\alpha_\xi a(\,\cdot\,,\xi)\right\|_{X}\right)\,.
\end{equation}
%

Tipically, $X=L^\infty(\R^N)$ or $X=LL(\R^N)$ for us. In the former case, for convenience we will use the notations
$\Gamma^{m+\delta\log}_\infty$, $\Sigma^{m+\delta\log}_\infty$ and $\|\,\cdot\,\|^{(m,\delta)}_{(\infty,k)}$.
In the final part of the present section (see the end of Paragraph \ref{sss:operators}), for the sake of generality we will consider
also the case $X=B^{\g+\vrho\log}_{\infty,\infty}(\R^N)$ (recall \eqref{eq:log-Besov} for its definition).

In the particular case $X=LL$, we explicitly notice the following fact: for $a\in\Gamma^{m+\delta\log}_{LL}$, there exists $K>0$ such that,
for all $\xi\in\R^N$ and all $y\in\R^N\!\setminus\!\{0\}$, one has
\begin{equation} \label{eq:LL_sem} 
\sup_{\R^N_x}\left|a(x+y,\xi)\,-\,a(x,\xi)\right|\,\leq\,K\,(1+|\xi|)^{m}\,
\log^{\delta}(2+|\xi|)\,|y|\,\log\!\left(1+\frac{1}{|y|}\right)\,.
\end{equation} 
Hence we can we set $|a|_{LL}\,=\,|a|^{(m,\delta)}_{(LL,0)}$ to be the smallest constant $K$ such that the previous inequality holds true.
In a quite natural way, we can also define the $LL$ seminorms $|a|^{(m,\delta)}_{(LL,k)}$.

When $X=B^{\g+\vrho\log}_{\infty,\infty}$, instead, we set $\Gamma^{m+\delta\log}_{\g+\vrho\log}\,:=\,\Gamma^{m+\delta\log}_{X}$.
Moreover, introducing a Littlewood-Paley decomposition $\bigl(\Delta_\nu\bigr)_{\nu\geq0}$ in the $x$-variable (not in $\xi$),
we have
\begin{equation} \label{eq:B_sem} 
\sup_{\R^N_x}\left|\Delta_\nu a(\,\cdot\,,\xi)\right|\,\leq\,K\,(1+|\xi|)^{m}\,
\log^{\delta}(2+|\xi|)\;2^{-\g\nu}\,(1+\nu)^{-\vrho}\,,
\end{equation} 
for a constant $K>0$, for all $\xi\in\R^N$ and all $\nu\geq0$.
Once again, in a natural way we can introduce the seminorms $\|a\|^{(m,\delta)}_{(\g+\vrho\log,k)}$.

Finally, we explicitly point out that, by spectral localization and Paley-Wiener Theorem, a symbol
$\sigma\in\Sigma^{m+\delta\log}_{X}$ is smooth also in the $x$ variable.

\medbreak
Now let us consider a symbol $a\in\Gamma^{m+\delta\log}_{X}$: we can associate to it a the classical symbol according to the formula
\begin{equation} \label{eq:classical-symb} 
\sigma^\psi_a(x,\xi)\,:=\,\left(\,\psi(D_x,\xi)\,a\,\right)(x,\xi)\,=\,\left(G^\psi(\cdot,\xi)\,*_x\,a(\cdot,\xi)\right)(x)\,.
\end{equation} 
The following proposition holds true.
\begin{prop} \label{p:par-op}
Let $(m,\delta)\in\R^2$.
\begin{itemize}
\item[(i)] For $X\subset L^\infty(\R^N)$ a Banach space, the smoothing operator $\mbb{S}:\,a(x,\xi)\,\mapsto\,\sigma_a(x,\xi)$
maps continuously $\Gamma^{m+\delta\log}_{X}$ into $\Sigma^{m+\delta\log}_{X}$.
\item[(ii)] For $a\in\Gamma^{m+\delta\log}_{LL}$, then the difference symbol $a\,-\,\sigma_a$ belongs to
$\Gamma^{(m-1)+(\delta+1)\log}_\infty$.
\item[(iii)] In particular, if $\psi_1$ and $\psi_2$ are two admissible cut-off functions, then the difference of the two
smoothed symbols, $\sigma^{\psi_1}_a\,-\,\sigma^{\psi_2}_a$, belongs to $\Sigma^{(m-1)+(\delta+1)\log}_{\infty}$.
\end{itemize}
\end{prop}

\begin{proof}
The first property is classical. The second one immediately follows from the log-Lipschitz continuity assumption
and Lemma \ref{l:G}. The last statement is a straightforward consequence of the previous ones.
\end{proof}

We conclude this  part by noting that, at this level, the time variable can be treated as a parameter
in the construction. In particular, the previous properties still hold true for symbols in
$L^\infty\bigl([0,T];\Gamma^{m+\delta\log}_{X}\bigr)$.

\subsubsection{Operators, symbolic calculus} \label{sss:operators}

As announced above, we can use the previous construction to associate to any symbol $a(x,\xi)$, which is non-regular in $x$,
an operator: this will be the pseudodifferential operator associated to the smooth symbol $\sigma^\psi_a$.

Let us formalize the discussion: for $a\in\Gamma^{m+\delta\log}_{\infty}$, we define the  paradifferential operator
associated to $a$ via the formula
$$
T^\psi_a\,:=\,\sigma^\psi_a(x,D_x)\,:\quad u\;\longmapsto\;T^\psi_au(x)\,=\,\frac{1}{(2\pi)^N}\int_{\R^N_\xi}e^{ix\cdot\xi}\,
\sigma^\psi_a(x,\xi)\,\what{u}(\xi)\,d\xi\,.
$$

\begin{rem} \label{r:p-prod}
Notice that, if $f=f(\xi)$ is a Fourier multiplier, then $T_f\,\equiv\,f(D_x)$ (see e.g. \cite{M-2008}).

Let us also point out that if $a=a(x)\in L^\infty$ and if we take the cut-off function $\psi_{-3}$, then $T^\psi_a$ is actually the
classical paraproduct operator, defined first in \cite{Bony}.
\end{rem}

Let us recall some basic definitions and properties in Paradifferential Calculus. The corresponding proofs in the logarithmic
setting are analogous to the classical case, and they are not detailed here.
\begin{defin} \label{d:op_order}
 We say that an operator $P$ is of order $\,m+\delta\log\,$ if, for every $(s,\alpha)\in\R^2$,
$P$ maps $H^{s+\alpha\log}$ into $H^{(s-m)+(\alpha-\delta)\log}$ continuously.
\end{defin}

With slight modifications to the proof of Proposition 2.9 of \cite{M-Z}, stated for the classical Sobolev class,
we get the next fundamental result.
\begin{lemma} \label{l:action}
 For all $\sigma\in\Sigma^{m+\delta\log}_{\infty}$, the corresponding operator $\sigma(\,\cdot\,,D_x)$ is of order $\,m+\delta\log$.
\end{lemma}
The following result is an immediate consequence of the previous lemma and Proposition \ref{p:par-op}.
\begin{thm} \label{t:action}
Given a symbol $a\in\Gamma^{m+\delta\log}_{\infty}$, for any admissible cut-off function $\psi$, the operator
$T^\psi_a$ is of order $m+\delta\log$.
\end{thm}

Notice that, in Lemma \ref{l:action} and Theorem \ref{t:action}, the $LL$ hypothesis is not necessary: these results hold true if the symbol
is even just $L^\infty$ with respect to $x$. On the contrary, we are going to exploit the additional regularity in space in the next result.
It states that the whole construction does not depends on the cut-off function $\psi$.
\begin{prop} \label{p:act-psi}
If $\psi_1$ and $\psi_2$ are two admissible cut-off functions and $a\in\Gamma^{m+\delta\log}_{LL}$,
then the difference $\,T^{\psi_1}_a\,-\,T^{\psi_2}_a\,$ is of order $(m-1)+(\delta+1)\log$.
\end{prop}
Therefore, changing the cut-off function $\psi$ doesn't change the paradifferential operator associated to $a$,
up to lower order terms. So, in what follows we fix the cut-off function $\psi=\psi_{-3}$ (defined in Paragraph \ref{sss:LL}) and
we will miss out the dependence of $\sigma_a$ and $T_a$ on it.

We want now to develop symbolic calculus in the $LL$ class. A preliminary result is in order: it can be viewed as a generalization
of Proposition \ref{p:dyadic-LL}.
\begin{lemma} \label{l:ll-symb}
Let $a\in\Gamma^{m+\delta\log}_{LL}$, and denote by $\sigma_a$ the classical symbol associated to it via formula
\eqref{eq:classical-symb}. Then the following estimates hold true:
\begin{eqnarray*}
 \left|\d^\alpha_\xi\sigma_{a}\right| & \leq & C_\alpha\left(1+|\xi|\right)^{m-|\alpha|}\,\log^\delta\left(2+|\xi|\right) \\
\left|\d^\beta_x\d^\alpha_\xi\sigma_{a}\right| & \leq & C_{\beta,\alpha}\left(1+|\xi|\right)^{m-|\alpha|+|\beta|-1}\,
\log^{\delta+1}\left(2+|\xi|\right)\,.
\end{eqnarray*}
The constants $C_\alpha$ just depend on the quantities $\|a\|^{(m,\delta)}_{\infty,k}$ defined in \eqref{eq:L-inf_sem},
where $|\alpha|\leq k$. \\
The constants $C_{\beta,\alpha}$, instead, depend only on the quantities $\|a\|^{(m,\delta)}_{LL,k}$, where again $|\alpha|\leq k$.
\end{lemma}
The proof of the previous result is somehow classical, as it follows the same lines of Lemma 3.16 in \cite{C-DS-F-M_tl}
and Lemma 3.15 in \cite{C-DS-F-M_wp}. Therefore we omit it.

From Lemma \ref{l:ll-symb} we immediately deduce the following properties.
\begin{thm} \label{t:symb_calc}
 \begin{itemize}
 \item[(i)] Let us take two symbols $a\in\Gamma^{m+\delta\log}_{LL}$ and $b\in\Gamma^{n+\vrho\log}_{LL}$ and denote by $T_a$, $T_b$ the
respective associated paradifferential operators. Then one has
$$ 
 T_a\,\circ\,T_b\,\,=\,\,T_{a\,b}\,\,+\,\,R_\circ\,.
$$ 
The principal part $T_{a\,b}$ is of order $(m+n)+(\delta+\vrho)\log$. \\
The remainder operator $R_\circ$ has order $(m+n-1)+(\delta+\vrho+1)\log$.
\item[(ii)] Let $a\in\Gamma^{m+\delta\log}_{LL}$.
The adjoint operator (over $L^2$) of $T_a$ is given by the formula
$$ 
 \left(T_a\right)^*\,\,=\,\,T_{\oline{a}}\,\,+\,\,R_*\,.
$$ 
The order of $T_{\oline{a}}$ is still $m+\delta\log$. \\
The remander operator $R_*$ has order $(m-1)+(\delta+1)\log$. 
 \end{itemize}
\end{thm}

The last statement of this paragraph is a fundamental paralinearization result, in the general instance $X\,=\,B^{\g+\vrho\log}_{\infty,\infty}$, which
will allow us to treat both the first and lower order terms of operator $L$ in energy estimates. In order to give sense to all terms, we have to restrict to
differential operators, and then operator $T_a$ reduces to the classical paraproduct operator (keep in mind also Remark \ref{r:p-prod}).
\begin{thm} \label{t:paralin}
Let $m\in\N$, and $\eta\in\N^N$ of lenght $|\eta|= m$. Take a pair $(\g,\vrho)\in\R^2$ such that $\g\geq0$, and consider a function
$a\in B^{\g+\vrho\log}_{\infty,\infty}$, introduced in \eqref{eq:log-Besov}. Define the difference operator
$\mc{D}:\,u\,\longmapsto\,a\,\d_x^\eta u\,-\,T_a\d_x^\eta u\,=\,a\,\d_x^\eta u\,-\,T_{a\,\xi^\eta}u$.
\begin{itemize}
 \item[(i)] If $s>m$, then $\mc D$ maps continuously $H^{s+\alpha\log}$ into
$H^{\g+(\vrho-h)\log}$, for any $h>1/2$.
 \item[(ii)] If $s=m$ and $\alpha\geq0$, the previous statement remains true.
\item[(iii)] For any $s\in\,]m-\g,m[\,$ (and then $\g>0$) and any $\alpha\in\R$, $\mc D$ maps continuously $H^{s+\alpha\log}$ into
$H^{\sigma+h\log}$, where $\sigma=s-m+\g$ and $h=\alpha-\delta+\vrho$.
\end{itemize}
The norms of the operators just depend on the quantity $\|a\|_{B^{\g+\vrho\log}_{\infty,\infty}}$.
\end{thm}

\begin{proof}
We start by noticing that, according to \eqref{eq:paraprod}, $\mc{D}(u)$ can be rewritten as
$$
\mc{D}(u)\,=\,\sum_{\nu\geq0}\d_x^\eta\left(S_{\nu+3}u\right)\,\Delta_\nu a\,=\,T'_{\d_x^\eta u}a\,=\,T_{\d_x^\eta u}a\,+\,R(\d_x^\eta u,a)\,.
$$

First of all, we focus on the case of high regularity, i.e. $s>m$, or $s=m$ and $\alpha\geq0$. This in particular implies that
$\d^\eta_xu\in L^2$, with $\|\d^\eta_xu\|_{L^2}\,\leq\,C\,\|u\|_{H^{s+\alpha\log}}$. Therefore, we can estimate
$$
\left\|S_{\nu-3}\d^\eta_xu\,\Delta_\nu a\right\|_{L^2}\,\leq\,C\,\left\|S_{\nu-3}\d^\eta_xu\right\|_{L^2}\,\left\|\Delta_\nu a\right\|_{L^\infty}
\,\leq\,C\,\|u\|_{H^{s+\alpha\log}}\,\|a\|_{B^{\g+\vrho\log}_{\infty,\infty}}\,2^{-\g\nu}\,(1+\nu)^{-\vrho}\,.
$$
Thanks to Lemma \ref{l:log-B_ind}, we immediately deduce that $T_{\d^\eta_xu}a\,\in\,H^{\g+(\vrho-h)\log}$ for all $h>1/2$.

As for $R$, once again we can focus just on the diagonal terms $R_0(\d^\eta_xu,a)\,=\,\sum_\nu\d^\eta_x\Delta_\nu u\,\Delta_\nu a$. For any $k\in\N$,
let us estimate
\begin{eqnarray*}
\left\|\Delta_kR_0(\d^\eta_xu,a)\right\|_{L^2} & \leq & C\,\sum_{\nu\geq k-3}\left\|\Delta_\nu\d^\eta_xu\right\|_{L^2}\,
\left\|\Delta_\nu a\right\|_{L^\infty}  \\
& \leq & C\,\|u\|_{H^{s+\alpha\log}}\,\|a\|_{B^{\g+\vrho\log}_{\infty,\infty}}\,\sum_{\nu\geq k-3}2^{\nu(m-s-\g)}\,
(1+\nu)^{-\alpha-\vrho}\,\zeta_\nu \\
& \leq & C\,\|u\|_{H^{s+\alpha\log}}\,\|a\|_{B^{\g+\vrho\log}_{\infty,\infty}}\,\left(\sum_{\nu\geq k-3}2^{2\nu(m-s-\g)}\,
(1+\nu)^{-2(\alpha+\vrho)}\right)^{1/2}\,,
\end{eqnarray*}
where, as usual, $\left\|\zeta_\nu\right\|_{\ell^2}=1$. Therefore, by Lemma \ref{l:log-B_ind} again, the remainder term
$R(\d^\eta_xu,a)\,\in\,H^{(s-m+\g)+(\alpha+\vrho-h)\log}$ for all $h>1/2$, which is in particular included in $H^{\g+(\vrho-h)\log}$.
When $s=m$, the same speech holds true for any $\alpha\geq0$.

Now let us turn our attention to the case $m-\g<s<m$ (and then $\g>0$): for this, we consider directly the operator $T'_{\d^\eta_xu}a$.
Since $s-m<0$, we can apply Lemma \ref{l:log-S_j} and estimate
$$
\left\|S_{\nu+3}\d^\eta_xu\,\Delta_\nu a\right\|_{L^2}\,\leq\,C\,\|u\|_{H^{s+\alpha\log}}\,2^{\nu(m-s)}\,(1+\nu)^{-\alpha}\,\zeta_\nu\,
\|a\|_{B^{\g+\vrho\log}_{\infty,\infty}}\,2^{-\nu\g}\,(1+\nu)^{-\vrho}\,,
$$
where the sequence $\bigl(\zeta_\nu\bigr)_\nu\in\ell^2$ is as above. Thanks to this inequality and the fact that
$s-m+\g>0$ by hypothesis, Lemma \ref{l:log-ball} implies that $T'_{\d^\eta_xu}a$ belongs to
$H^{(s-m+\g)+(\alpha+\vrho)\log}$.

The proof of the theorem is now completed.
\end{proof}

We notice the following fact: if $a\in W^{1,\infty}$ is Lipschitz, it is well-known that $au-T_au\,\in\,H^1$ for any $u\in L^2$
(see e.g. Theorem 5.2.8 of \cite{M-2008}). On the contrary, if we applied our result with $m=s=\alpha=\vrho=0$ and $\g=1$,
we would get just $au-T_au\,\in\,H^{1-(1/2+\delta)\log}$ for any $\delta>0$. 

Motivated by this consideration, let us make a remark.

\begin{rem} \label{r:paralin}
Attaining the limit case $h=0$ in points \textit{(i)}, \textit{(ii)}, and $s=m-\g$ in point \textit{(iii)},
would require further technical extensions of the theory, in the same spirit of Paragraph 5.2.4
of \cite{M-2008}, to functions $a$ in logarithmic H\"older classes $B^{\g+\vrho\log}_{\infty,\infty}$, where $\vrho\neq0$.

However, these adaptations go beyond the scopes of the present paper, and we decided not to address these issues here in order to keep
the presentation as coincise as possible.
Indeed, due to product properties (see Proposition \ref{p:Hol-Sob}) and loss of derivatives in the
energy estimates, we will always be far away from these endpoint cases, and the previous statement turns to be enough for our scopes.
\end{rem}

Before going further, let us remark that, at this level, time can be treated once again as a parameter in the construction: for a symbol
$a=a(t,x,\xi)\in L^\infty\bigl([0,T];\Gamma^{m+\delta\log}_{\infty}\bigr)$ and $u\in\mc{S}'([0,T]\times\R^d)$, we set
$$
\left(T_au\right)(t,\,\cdot\,)\,:=\,T_{a(t,\,\cdot\,,\xi)}u(t,\,\cdot\,)\,.
$$
As a consequence, all the properties stated above still hold true for the time-dependent operator, at any time $t\in[0,T]$ fixed.
Indeed, only regularity in space is used at this level. 

\subsubsection{The case of symbols which are log-Lipschitz in time} \label{sss:LL-t}

Now, we get closer to our hypotheses, and we introduce a new class of symbols, by imposing additional regularity in the time variable.
\begin{defin} \label{d:symbol_t}
Let $Y\,\subset\,L^\infty([0,T])$ a Banach space.
For $(m,\delta)\in\R^2$, we define ${Y}_T\bigl(\Gamma^{m+\delta\log}_{X}\bigr)$ as the class of symbols
$a(t,x,\xi)\,\in\,L^\infty\bigl([0,T];\Gamma^{m+\delta\log}_{X}\bigr)$ such that,
for almost every $(x,\xi)\in\R^N\times\R^N$, the map $\;t\,\mapsto\,a(t,x,\xi)\;$ belongs to $Y$. 

Analogously, we define the class $Y_T\bigl(\Sigma^{m+\delta\log}_{X}\bigr)$ if, moreover, the spectral condition (in $x$)
of Definition \ref{d:symbols} is verified for almost every $(t,\xi)\in[0,T]\times\R^N$.

We omit the subscript $T$ whenever $T=+\infty$.
\end{defin}

For us $Y_T$ will be always $L^\infty([0,T])$ or $LL([0,T])$.  

In particular, when $Y=LL$, the previous definition implies that there exists a $C_0>0$
for which, for all $(x,\xi)\in\R^N\times\R^N$ and all $0<\tau\leq T/2$, one has
$$ 
\left|a(t+\tau,x,\xi)\,-\,a(t,x,\xi)\right|\,\leq\,C_0\,(1+|\xi|)^{m}\,\log^{\delta}(2+|\xi|)\,\tau\,\log\!\left(1+\frac{1}{\tau}\right)\,.
$$ 
As before, for such a symbol $a$ we can then define the seminorm $|a|_{LL_t}$. Of course, higher order seminorms (related to derivatives
in $\xi$) can be defined, but they will not be used in our study, so that we prefer to limit the presentation to this case.

\begin{rem} \label{r:LL}
 We remark that the $LL$ continuity of the symbol, separately with respect to time and space variables, is enough to our scopes.
This will be evident from our computations.
\end{rem}

In order to perform energy estimates, for a symbol $a$ as above we need to introduce a regularization in time.
So, take an even function $\rho\in\mc{C}^\infty_0(\mbb{R})$, $0\leq\rho\leq1$,
whose support is contained in the interval $[-1,1]$ and such that $\int\rho(t)dt=1$, and define the mollifier kernel
$$
\rho_\veps(t)\,:=\,\frac{1}{\veps}\,\,\rho\!\left(\frac{t}{\veps}\right)\qquad\qquad\forall\,\veps\in\,]0,1]\,.
$$
Let us fix a symbol $a\in LL_T\bigl(\Gamma^{m+\delta\log}_X\bigr)$; if $T<+\infty$, we extend this symbol out of $[0,T]$ (for instance, by taking the constant values
at the extremities of the interval), in such a way to get a new symbol (which we will still denote by $a$) in the class $LL\bigl(\Gamma^{m+\delta\log}_X\bigr)$.
Now, we smooth $a$ out setting, for all $\veps\in\,]0,1]$,
\begin{equation} \label{eq:a^e}
a_\veps(t,x,\xi)\,:=\,\bigl(\rho_\veps\,*_t\,a(\,\cdot\,,x,\xi)\bigr)(t)\,=\,\int_{\mbb{R}_s}\rho_{\veps}(t-s)\,a(s,x,\xi)\,ds\,.
\end{equation}
Then, we have the following estimates (see e.g. \cite{C-DG-S}, \cite{C-L}).
\begin{lemma} \label{l:LL-reg}
 Let $a\in LL\bigl(\Gamma^{m+\delta\log}_X\bigr)$. Then $\bigl(a_\veps\bigr)_\veps$ is a bounded family in the same space.
In particular, Proposition \ref{p:par-op} still holds true for $a_\veps$, uniformly in
$\veps\in\,]0,1]$.
 
In addition, there exist constants $C>0$ such that, for all $\veps\in\,]0,1]$ and for all $(t,x,\xi)\in\R\times\R^N\times\R^N$,
one has
\begin{eqnarray*}
\left|a_\veps(t,x,\xi)-a(t,x,\xi)\right| & \leq & C\,|a|_{LL_t}\,\,\veps\,\log\left(1+\frac{1}{\veps}\right)\,
(1+|\xi|)^{m}\,\log^{\delta}(2+|\xi|) \\
\left|\d_ta_\veps(t,x,\xi)\right| & \leq & C\,|a|_{LL_t}\,\log\!\left(1+\frac{1}{\veps}\right)\,(1+|\xi|)^{m}\,\log^{\delta}(2+|\xi|)\,.
\end{eqnarray*}
\end{lemma}

In the course of the proof, it will be fundamental to link the approximation parameter $\veps$ with the dual variable,
following the original idea of \cite{C-DG-S} (see also \cite{C-DS-F-M_tl}, \cite{C-DS-F-M_wp}, \cite{C-DS-F-M_Birk}, \cite{C-DS-F-M_Z-sys},
\cite{C-L}, \cite{C-M}, \cite{Tar}). More precisely, in \cite{C-DG-S} the authors took $\veps\,=\,1/|\xi|$:
here, we make an analogous choice, but replacing $|\xi|$ by $\lan\xi\ran\,:=\,\left(1+|\xi|^2\right)^{1/2}$ (we need the new symbol to be defined
for all $\xi\in\R^N$).
\begin{prop} \label{p:Y-tilde}
For $a\in Y_T\bigl(\Gamma^{m+\delta\log}_X\bigr)$, where $Y\subset L^\infty([0,T])$ is a Banach space, we define the new function
$$
\wtilde{a}(t,x,\xi)\,:=\,a_{1/\lan\xi\ran}(t,x,\xi)\,.
$$
Then $\wtilde{a}$ is still a symbol in the class $Y_T\bigl(\Gamma^{m+\delta\log}_X\bigr)$, which is actually
smooth in time.
\end{prop}

\begin{proof}
Proposition 3.22 of \cite{C-DS-F-M_Z-sys} proves that the convolution acts as an operator of order $0+0\log$ in the time variable.
We are going to show that this is true also with respect to $\xi$.

First of all, we notice that it is enough to consider the case $a\in L^\infty_T\bigl(\Gamma^{0+0\log}_\infty\bigr)$:
regularity in time comes from Lemma \ref{l:LL-reg}, while the $x$ variable does not play any role at this level.

For $\alpha=0$, the estimate of Definition \ref{d:symbols} is immediate. Hence, let us consider the case $|\alpha|=1$: by definitions
we have
\begin{eqnarray*}
& & \hspace{-0.5cm}
\d^\alpha_\xi\wtilde{a}(t,x,\xi)\,=\,\d^\alpha_\xi\lan\xi\ran\int\rho\bigl((t-s)\lan\xi\ran\bigr)\,a(s,x,\xi)\,ds\,+ \\
& & \qquad\quad+\,\lan\xi\ran\,\d^\alpha_\xi\lan\xi\ran\int\rho'\bigl((t-s)\lan\xi\ran\bigr)\,(t-s)\,a(s,x,\xi)\,ds\,+\,
\lan\xi\ran\int\rho\bigl((t-s)\lan\xi\ran\bigr)\,\d^\alpha_\xi a(s,x,\xi)\,ds\,.
\end{eqnarray*}
The estimates for the first and last term are straightforward: thanks to the uniform bounds
$$
\left|\d^\alpha_\xi a(t,x,\xi)\right|\,\leq\,C_\alpha\,\left(1+|\xi|\right)^{-|\alpha|}\,,
$$
(recall that we are taking $a\in L^\infty_T\bigl(\Gamma^{0+0\log}_\infty\bigr)$ only), we have to control
$$
\int\rho\bigl(t\,\lan\xi\ran\bigr)\,dt\,\leq\,C\,/\,\lan\xi\ran\,,
$$
where we used the change of variable $\tau\,=\,\lan\xi\ran\,t$ and fact that $\int\rho\,\equiv\,1$.

For the second term, we argue as above: by Young inequality, we are reconducted to consider
the $L^1$ norm of the function $\left|\rho'\bigl(\lan\xi\ran\,t\bigr)\right|\,\lan\xi\ran\,|t|$. Performing the same change of variable as before,
we get
$$
\lan\xi\ran\int\left|\rho'\bigl(t\,\lan\xi\ran\bigr)\right|\,|t|\,dt\,\leq\,C\,/\,\lan\xi\ran\,,
$$
and this concludes the proof of the proposition.
\end{proof}

Let us now focus on the case $Y=LL$. Up to extend the symbols out of $[0,T]$ as explained above, without loss of generality we can focus
just on the case $T=+\infty$.

\begin{prop} \label{p:LL-tilde}
Let $a\in LL\bigl(\Gamma^{m+\delta\log}_X\bigr)$. Then the following properties hold true.
\begin{itemize}
\item[(i)] One has $\;a-\wtilde{a}\,\in\,L^\infty\bigl(\Gamma^{(m-1)+(\delta+1)\log}_X\bigr)\;$ and
$\;\d_t\wtilde{a}\,\in\,L^\infty\bigl(\Gamma^{m+(\delta+1)\log}_X\bigr)$.
\item[(ii)] The smooothed (in $x$) symbol $\sigma_{\wtilde{a}}$, defined by  formula \eqref{eq:classical-symb},
belongs to $LL\bigl(\Sigma^{m+\delta\log}_X\bigr)$, and one has the ``commutation'' formula
$\wtilde{\sigma_a}\,=\,\sigma_{\wtilde{a}}$. So, we adopt the notation $\wtilde{\sigma}_a$ for it.
\item[(iii)] The smoothing operator (in time and space)  $\wtilde{\mbb{S}}:\,a(t,x,\xi)\,\mapsto\,\wtilde{\sigma}_a(t,x,\xi)$
maps continuously $LL\bigl(\Gamma^{m+\delta\log}_{X}\bigr)$ into $LL\bigl(\Sigma^{m+\delta\log}_{X}\bigr)$.
\item[(iv)] The symbol $\wtilde{\sigma}_a$ still satisfies the estimates of Lemma \ref{l:ll-symb}, uniformly in $t$. In addition,
one  has $\sigma_a-\wtilde{\sigma}_a\,\in\,L^\infty\bigl(\Sigma^{(m-1)+(\delta+1)\log}_X\bigr)$ and
$\sigma_{\d_t\wtilde{a}}\,=\,\d_t\wtilde{\sigma}_a\,\in\,L^\infty\bigl(\Sigma^{m+(\delta+1)\log}_X\bigr)$,
with the estimates, for some constants $C$ just depending on $|a|_{LL_t}$,
\begin{eqnarray*}
\left|\sigma_a(t,x,\xi)-\wtilde{\sigma}_a(t,x,\xi)\right| & \leq & C\,
(1+|\xi|)^{m-1}\,\log^{\delta+1}(2+|\xi|) \\
\left|\d_t\wtilde{\sigma}_a(t,x,\xi)\right| & \leq & C\,(1+|\xi|)^{m}\,\log^{\delta+1}(2+|\xi|)\,,
\end{eqnarray*}
and analogous formula for the higher order derivatives in $\xi$.
\end{itemize}
\end{prop}

\begin{proof}
Assertion \textit{(i)} comes from estimates in Lemma \ref{l:LL-reg}, while point \textit{(ii)} is a direct consequence of
formula \eqref{eq:classical-symb} and definition of $\wtilde{a}$ (the integrals of the convolutions in time and space commute between
themselves).

Next, we notice that $\wtilde{\mbb{S}}$ is a composition of the self-map $a\,\mapsto\,\wtilde{a}$ of
$LL\bigl(\Gamma^{m+\delta\log}_{X}\bigr)$ with the map $\mbb{S}:\wtilde{a}\,\mapsto\,\wtilde{\sigma}_a$ which goes from the previous
space into $LL\bigl(\Sigma^{m+\delta\log}_{X}\bigr)$: sentence \textit{(iii)} follows.

Let us focus on assertion \textit{(iv)}: it basically relies on point \textit{(i)} of the present proposition, and in particular on 
the fact that the convolution acts as an operator of order $0$.
Then, the first statement is immediate, the second one is implied by the linearity of operator $\mbb{S}$ together with Lemma
\ref{l:LL-reg}. The formula for the time derivative derives directly from the definitions. Finally,
for the estimates one has to use Lemma \ref{l:LL-reg} again.
\end{proof}

Now, given a symbol $a\in Y\bigl(\Gamma^{m+\delta\log}_X\bigr)$, we define the paradifferential operator
\begin{equation} \label{eq:T-tilde}
\wtilde{T}_a\,:=\,T_{\wtilde{a}}\,=\,\wtilde{\sigma}_a(\,\cdot\,,D_x)\,.
\end{equation}
Notice that properties \textit{(ii)}, \textit{(iii)} of Proposition \ref{p:LL-tilde} still hold true even if we replace
$LL$ by a generic $Y$, and in particular when $Y=L^\infty$. Therefore, fixing $X=LL(\R^N)$,
from Theorem \ref{t:symb_calc} we immediately get the following result.
\begin{thm} \label{t:symb_tilde}
\begin{itemize}
 \item[(i)] If $a\in L^\infty\bigl(\Gamma^{m+\delta\log}_{LL}\bigr)$, then the operator $\wtilde{T}_a$ is of order $m+\delta\log$.
\item[(ii)] For $a\in L^\infty\bigl(\Gamma^{m+\delta\log}_{LL}\bigr)$ and $b\in L^\infty\bigl(\Gamma^{n+\vrho\log}_{LL}\bigr)$,
one has $\wtilde{T}_a\circ\wtilde{T}_b\,=\,T_{\wtilde{a}\,\wtilde{b}}\,+\,\wtilde{R}_\circ$, where
$\wtilde{R}_\circ$ is of order $(m+n-1)+(\delta+\vrho+1)\log$.
\item[(iii)] For $a\in L^\infty\bigl(\Gamma^{m+\delta\log}_{LL}\bigr)$, denote by $\wtilde{T}^*_a$ the adjoint
of $\wtilde{T}_a$ over $L^2$. Then $\wtilde{T}^*_a\,=\,\wtilde{T}_{\oline{a}}\,+\,\wtilde{R}_*$, where
$\wtilde{R}_*$ is of order $(m-1)+(\delta+1)\log$.
\end{itemize}
\end{thm}

\begin{rem} \label{r:comp}
Thanks to point \textit{(iv)} of Proposition \ref{p:LL-tilde}, in point \textit{(ii)} above we can substitute
$T_{\wtilde{a}\,\wtilde{b}}$ by $\wtilde{T}_{a\,b}$, up to another remainder which is still of order $(m+n-1)+(\delta+\vrho+1)$.
But this would require regularity in time, that we do not want to use at this level.
\end{rem}

In the same way, we can see that also the analogous of Theorem \ref{t:paralin} holds true for the operator
$\wtilde{\mc D}\,:=a\,-\,\wtilde{T}_a$, whenever $a\in L^\infty\bigl(\Gamma^{m+\delta\log}_{\g+\vrho\log}\bigr)$.

Finally, we take $Y=LL$ and $X=L^\infty$ and we exploit regularity in time. Still by use of Proposition \ref{p:LL-tilde}, we get the
following statement.
\begin{thm} \label{t:symb_time}
Let $a\in LL\bigl(\Gamma^{m+\delta\log}_{\infty}\bigr)$. Then the next properties are true:
\begin{itemize}
 \item the operator $T_a\,-\,\wtilde{T}_a$ is a remainder of order $(m-1)+(\delta+1)\log$;
\item one has $\bigl[\d_t,\wtilde{T}_a\bigr]\,=\,T_{\d_t\wtilde{a}}$, and this is an operator of order $m+(\delta+1)\log$.
\end{itemize}
\end{thm}

\section{Well-posedness on the whole $\R^n$} \label{s:en-est}
Thanks to the tools developed in the previous section, we are now able to prove Theorem \ref{th:en_LL}, i.e. energy estimates for the
global in space problem. The first part of the present section is devoted to this.
In the estimates, we will keep track of the dependence of the different constants on the log-Lipschitz seminorms of the coefficients
of the operator and of the symmetrizer.

In the final part (see Subsection \ref{ss:global_e}), we will show how to derive Theorem \ref{t:global_e} from the bounds of Theorem \ref{th:en_LL}.

\subsection{The energy} \label{ss:energy}
Let us start by defining the energy associated to our operator $L$. Roughly speaking, denoting by $S$ a microlocal symmetrizer for $\mc{A}$,
the leading idea is that the paradifferential operator $T_S$ is an approximated symmetrizer for $iT_{\mc A}$,
which represents the principal part of the operator $L$.

However, some ``corrections'' are needed. Indeed, on the one hand we need to smooth out the coefficients with respect to time
in order to be able to perform energy estimates. On the other hand, $S$ is homogeneous of degree $0$ in $\xi$, and in particular it can be
singular in $\xi=0$: therefore, we have to cut off the low frequencies, borrowing somehow an idea from paradifferential calculus
with parameters  (see e.g. \cite{M-1986}, \cite{M-Z}).

So, let us proceed in the following way. First of all, given a symmetrizer $S(t,x,\xi)$ for our system, we smooth it out with respect
to  the time variable, according to the formula \eqref{eq:a^e}. Notice that the approximated symmetrizer $S_\veps$
still satisfies
$$
0\,<\,\lambda\,\Id\,\leq\,S_\veps(t,x,\xi)\,\leq\,\Lambda\,\Id
$$
for any $\veps\in\,]0,1]$ and all $(t,x,\xi)\in[0,T]\times\R^n_x\times\left(\R^n_\xi\setminus\{0\}\right)$. In particular, the matrix symbol
$S_\veps^{1/2}$ is well-defined.

Furthermore, let us immediately set (accordingly with the original choice of \cite{C-DG-S})
$$
\veps\,=\,1/|\xi|\qquad\qquad \forall\;|\xi|\,\geq\,1\,.
$$
Notice that, since $|\xi|\,\geq\,1$, it makes no special difference to take $|\xi|$ or $\lan\xi\ran$.
Hence, in what follows we will adopt the notations introduced in Paragraph \ref{sss:LL-t}.

Finally, let $\theta\,\in\,\mc{C}^\infty_0(\R^n)$ such that $0\leq\theta\leq1$, $\theta\equiv1$ in the ball $B(0,1)$
and $\theta\equiv0$ for $|\xi|\geq2$. For any $\mu>0$, we set $\theta_\mu(\xi)\,:=\,\theta(\mu^{-1}\xi)$, and we denote
$$
\wtilde{\Sigma}(t,x,\xi)\,:=\,\wtilde{S}^{1/2}(t,x,\xi)\,\bigl(1-\theta_\mu(\xi)\bigr)\,.
$$

For all $(s,\alpha)\in\R^2$ fixed, we then define the quantities
$$
E_{s,\alpha}[u]\,:=\,\left\|\wtilde{T}_{\Sigma}u\right\|^2_{H^{s+\alpha\log}}\,+\,\left\|\theta_\mu(D_x)u\right\|^2_{H^{s+\alpha\log}}\,,
$$
where $\wtilde{T}_{\Sigma}$ is the paradifferential operator associated to the just defined matrix symbol $\wtilde{\Sigma}$,
according to the construction explained in Subsection \ref{ss:paradiff} (see in particular Paragraph \ref{sss:LL-t}).

We now establish positivity estimates: we show that this property does not depend on the considered Sobolev norm.
\begin{lemma} \label{l:energy}
\begin{itemize}
\item[(i)] There exists a $C_0>0$ (just depending on the constant $K_0$ appearing in condition \eqref{hyp:bound}
for $S$) such that, for any $(s,\alpha)\in\R^2$ and any $u\in\mc{S}(\R^n)$,
$$
E_{s,\alpha}[u]\,\leq\,C_0\,\|u\|_{H^{s+\alpha\log}}\,.
$$
\item[(ii)] There exists $\mu_0\geq2$ such that, for all $\mu\geq\mu_0$, the following property holds true:
there exists a $C_\mu>0$ for which, for all smooth $u\in\mc{S}(\R^n)$,
$$
\|u\|_{H^{s+\alpha\log}}\,\leq\,C_\mu\,E_{s,\alpha}[u]\,.
$$
The constant $C_\mu$ depends just on $\mu$, on $K_0$ and on $K_1$ (recall conditions \eqref{hyp:bound} and \eqref{hyp:LL} for $S$);
in particular, it is independent of $(s,\alpha)$.
\end{itemize}
\end{lemma}

\begin{proof}
The first property is immediate, once noticing that $\wtilde{\Sigma}$ is an operator of order $0+0\log$, in the sense of Definition
\ref{d:op_order}.

So, let us focus on the second inequality: according to the decomposition
$$
u\,=\,\bigl(1-\theta_\mu(D_x)\bigr)\,u\,+\,\theta_\mu(D_x)\,u\,,
$$
it is enough to prove it for the high frequency component $\bigl(1-\theta_\mu(D_x)\bigr)\,u$.

First of all, let us define $\psi_\mu(\xi)\,:=\,\bigl(1-\theta_\mu(\xi)\bigr)\,\bigl(1-\theta(\xi)\bigr)$: by the properties
of the support of $\theta$, we easily infer that, for any $\mu\geq2$, one  has $\psi_\mu\,\equiv\,\bigl(1-\theta_\mu\bigr)$.
Therefore, denoting
$$
\wtilde{\Xi}(t,x,\xi)\,:=\,\wtilde{S}^{-1/2}(t,x,\xi)\,\bigl(1-\theta(\xi)\bigr)\,,
$$
we deduce the equality $\wtilde{\Xi}\;\wtilde{\Sigma}\,=\,\Id\,\psi_\mu\,=\,\Id\,(1-\theta_\mu)$, which in turn gives,
by symbolic calculus (recall Theorem \ref{t:symb_tilde}),
\begin{equation} \label{eq:pos_high}
\bigl(1-\theta_\mu(D_x)\bigr)\,u\,=\,T_{\wtilde{\Xi}\,\wtilde{\Sigma}}u\,=\,T_{\wtilde{\Xi}}\,T_{\wtilde{\Sigma}}u\,+\,
\wtilde{R} u\,,
\end{equation}
where the remainder $\wtilde R$ is given, at the principal order, by $\d_\xi\wtilde{\Xi}\,\d_x\wtilde{\Sigma}$. Using
again the function $\psi_\mu$ defined above and introducing the matrix $\wtilde{S}'\,=\,\wtilde{S}^{1/2}\,\bigl(1-\theta\bigr)$,
it is easy to see that
$$
\wtilde{R}\,u\,=\,\wtilde{R}'\,\bigl(1-\theta_\mu(D_x)\bigr)\,u\,,\qquad\qquad\mbox{ with }\qquad
\wtilde{R}'\,\sim\,\d_\xi\wtilde{\Xi}\,\d_x\wtilde{S}'\,.
$$

Let us come back to \eqref{eq:pos_high}: since $\wtilde{\Xi}$ is of order $0+0\log$ and $\wtilde{R}'$ is of order $-1+\log$, we obtain
$$
\left\|\bigl(1-\theta_\mu(D_x)\bigr)\,u\right\|_{H^{s+\alpha\log}}\,\leq\,C_1\left\|\wtilde{T}_{\Sigma}u\right\|_{H^{s+\alpha\log}}\,+\,
C_2\left\|\bigl(1-\theta_\mu(D_x)\bigr)\,u\right\|_{H^{(s-1)+(\alpha+1)\log}}\,,
$$
for some $C_1$ (depending on $|\wtilde S|^{(0,0)}_{L^\infty,0}$) and $C_2$ (depending also on $|\wtilde S|^{(0,0)}_{LL,0}$) large enough.
Now, by spectral properties and Proposition \ref{p:log-H}, we have
\begin{eqnarray*}
\left\|\bigl(1-\theta_\mu(D_x)\bigr)\,u\right\|^2_{H^{(s-1)+(\alpha+1)\log}} & \leq &
\sum_{k\geq K_\mu}2^{2(s-1)k}\,(1+k)^{2(\alpha+1)}\,\left\|\Delta_k\bigl(1-\theta_\mu(D_x)\bigr)\,u\right\|^2_{L^2} \\
& \leq & 2^{-2K_\mu}\,(1+K_\mu)^2\,\left\|\bigl(1-\theta_\mu(D_x)\bigr)\,u\right\|_{H^{s+\alpha\log}}\,,
\end{eqnarray*}
where $K_\mu\,\sim\,\log_2\mu$.
Therefore, if $\mu$ is large enough, such that $C_2\,2^{-K_\mu}\,(1+K_\mu)\,\leq\,1/2$, we can abosorbe this term in the  left-hand
side of the previous inequality. This complete the proof of the statement \textit{(ii)}, and so also of the lemma.
\end{proof}

Now, let $s\in\,]0,1[\,$ be fixed. For $\beta>0$, to be chosen in the course of the proof, we set
\begin{equation} \label{eq:s}
s(t)\,=\,s\,-\,\beta\,t\,; 
\end{equation}
we then define the energies
\begin{eqnarray}
E(t) & := & E_{s(t),0}[u(t)]\;=\;\left\|\wtilde{T}_{\Sigma}u(t)\right\|^2_{H^{s(t)}}\,+\,\left\|\theta_\mu(D_x)u(t)\right\|_{H^{s(t)}}^2
\label{eq:E} \\
E_{\log}(t) & := & E_{s(t),1/2}[u(t)]\;=\;\left\|\wtilde{T}_{\Sigma}u(t)\right\|^2_{H^{s(t)+(1/2)\log}}\,+\,
\left\|\theta_\mu(D_x)u(t)\right\|^2_{H^{s(t)+(1/2)\log}}  \label{eq:E_log}\,,
\end{eqnarray}
where $\mu$ is fixed large enough, so that Lemma \ref{l:energy} holds true. In particular, we have
$$
E(t)\,\sim\,\left\|u(t)\right\|^2_{H^{s(t)}}\qquad\qquad \mbox{ and } \qquad\qquad
E_{\log}(t)\,\sim\,\left\|u(t)\right\|^2_{H^{s(t)+(1/2)\log}}\,.
$$

\begin{rem} \label{r:energy}
Let us point out that the ``true'' energy associated to our operator is $E$. The second energy $E_{\log}$ is introduced because of a
logarithmic loss of derivatives in the estimates for $E$, due to the log-Lipschitz regularity of the coefficients and of the symmetrizer.
\end{rem}

We conclude this part with an approximation result: we make a paralinearization of our operator. In this way, we create remainders, which nonetheless
are regular enough.
\begin{lemma} \label{l:L->T}
Let $L$ be the operator defined by \eqref{def:L}, and $\wtilde{L}$ be given by \eqref{def:L*}. Then
\begin{equation} \label{eq:paralin}
Lu\,=\,\d_tu\,+\,i\,T_{\mc{A}}u\,+\,T_Bu\,+\,\mc{R}_Lu\,,
\end{equation}
where $\mc{R}_L$ maps continuously $H^{s+\alpha\log}$ into $H^{s+(\alpha-1)\log}$ for all $0<s<\g$ and all $\alpha\in\R$.

Analogously, we can write $\wtilde{L}u\,=\,\d_tu\,+\,i\,T_{\mc{A}}u\,+\,T_Bu\,+\,\mc{R}_{\wtilde{L}}u$, where $\mc{R}_{\wtilde L}$
is continuous from $H^{s+\alpha\log}$ to $H^{s+(\alpha-1)\log}$ for all $-\g<s<0$ and all $\alpha\in\R$.

The operator norms of both $\mc{R}_L$ and $\mc{R}_{\wtilde L}$ depend only on the constants $K_0$, $K_1$ and $K_2$ appearing in conditions
\eqref{hyp:bound}-\eqref{hyp:LL}-\eqref{hyp:Holder}.
\end{lemma}

\begin{proof}
Let us focus on operator $L$ first: decomposition \eqref{eq:paralin} is easily obtained, if we set
$\mc{R}_L\,=\,\mc{R}_A+\mc{R}_B$, where we have defined
$$
\mc{R}_A\,u\,:=\,\sum_{j=1}^n\left(A_j-T_{A_j}\right)\d_ju\qquad\mbox{ and }\qquad
\mc{R}_B\,u\,:=\,\left(B-T_B\right)u\,.
$$
We observe that $\mc{R}_A(t,x,\xi)\,=\,\sum_j\bigl(A_j(t,x)-T_{A_j}(t,x)\bigr)\,\xi_j$ is a symbol in the class
$L^\infty_T\bigl(\Gamma^{1+0\log}_{LL}\bigr)$: since $LL\hookrightarrow B^{1-\log}_{\infty,\infty}$, Theorem \ref{t:paralin}, point
\textit{(iii)}, gives us
$$
\left\|\mc{R}_A\,u\right\|_{H^{s+(\alpha-1)\log}}\,\leq\,C\,\left(\sup_{t\in[0,T]}\sup_{1\leq j\leq n}\|A_j(t,\,\cdot\,)\|_{LL}\right)\,
\left\|u\right\|_{H^{s+\alpha\log}}
$$
for any $0<s<1$ and any $\alpha\in\R$. In particular, this is true for $0<s<\g$.

On the other hand, $\mc{R}_B(t,x,\xi)$ belongs to $L^\infty_T\bigl(\Gamma^{0+0\log}_{\g+0\log}\bigr)$
(actually, it  does not even depend on $\xi$). Then, point \textit{(i)} of Theorem \ref{t:paralin} (with e.g. the particular choice
$\sigma=s<\g$) implies
$$
\left\|\mc{R}_B\,u\right\|_{H^{s+\alpha\log}}\,\leq\,C\,\left(\sup_{t\in[0,T]}\|B(t,\,\cdot\,)\|_{\mc{C}^\g}\right)\,
\left\|u\right\|_{H^{s+\alpha\log}}\,.
$$
This inequality completes the proof of the statement for $\mc{R}_L$.

\medbreak
Let us now deal with $\wtilde{L}$: we start by observing that
\begin{eqnarray*}
\wtilde{L}u & = & \d_tu\,+\,\sum_{j=1}^n\d_j\Bigl(\bigl(A_j\,-\,T_{A_j}\bigr)\,u\Bigr)\,+\,\sum_{j=1}^n\d_j\left(T_{A_j}u\right)\,+\,
T_Bu\,+\,\mc{R}_Bu \\
& = & \d_tu\,+\,i\,T_{\mc{A}}u\,+\,\sum_{j=1}^n\d_j\Bigl(\bigl(A_j\,-\,T_{A_j}\bigr)\,u\Bigr)\,+\,
\sum_{j=1}^n\left[\d_j,T_{A_j}\right]u\,+\,T_Bu\,+\,\mc{R}_Bu\,.
\end{eqnarray*}
Then, we just set $\mc{R}_{\wtilde{L}}\,:=\,\sum_{j=1}^n\d_j\Bigl(\bigl(A_j\,-\,T_{A_j}\bigr)\,u\Bigr)\,+\,
\sum_{j=1}^n\left[\d_j,T_{A_j}\right]u\,+\,\mc{R}_Bu$. Let us consider each of its terms one by one.

$\mc{R}_B$ is defined as before: this time, we apply item \textit{(iii)} of Theorem \ref{t:paralin} and we get,
for any $-\g<s<0$ and any $\alpha\in\R$,
$$
\left\|\mc{R}_B\,u\right\|_{H^{(s+\g)+\alpha\log}}\,\leq\,C\,\left(\sup_{t\in[0,T]}\|B(t,\,\cdot\,)\|_{\mc{C}^\g}\right)\,
\left\|u\right\|_{H^{s+\alpha\log}}\,.
$$

As for the commutator term, we notice that $\left[\d_j,T_{A_j}\right]\,=\,\d_j\sigma_{A_j}(t,x,D_x)$, where $\sigma_{A_j}(t,x,\xi)$
is the classical symbol  associated to $A_j$ via formula \eqref{eq:classical-symb}. Therefore, by Lemma \ref{l:ll-symb}
we deduce that $\left[\d_j,T_{A_j}\right]$ is an operator of order $0+\log$, and then, for any $(s,\alpha)\in\R^2$,
$$
\left\|\left[\d_j,T_{A_j}\right]u\right\|_{H^{s+(\alpha-1)\log}}\,\leq\,C\,
\left(\sup_{t\in[0,T]}\sup_{1\leq j\leq n}\|A_j(t,\,\cdot\,)\|_{LL}\right)\,\left\|u\right\|_{H^{s+\alpha\log}}\,.
$$

Finally, for any $j$, the operator $A_j-T_{A_j}$ belongs to the class $L^\infty_T\bigl(\Gamma^{0+0\log}_{LL}\bigr)$:
then again, point \textit{(iii)} of Theorem \ref{t:paralin} gives us
\begin{eqnarray*}
\left\|\d_j\Bigl(\bigl(A_j\,-\,T_{A_j}\bigr)\,u\Bigr)\right\|_{H^{s+(\alpha-1)\log}} & \leq & C\,
\left\|\bigl(A_j\,-\,T_{A_j}\bigr)\,u\right\|_{H^{(1+s)+(\alpha-1)\log}} \\
& \leq & C\,\left(\sup_{t\in[0,T]}\sup_{1\leq j\leq n}\|A_j(t,\,\cdot\,)\|_{LL}\right)\,\|u\|_{H^{s+\alpha\log}}
\end{eqnarray*}
for any $-1<s<0$, and in particular for $-\g<s<0$.

This completes the proof of the assertion for $\mc{R}_{\wtilde L}$, and then also of the lemma.
\end{proof}

\subsection{Energy estimates: proof of Theorem \ref{th:en_LL}} \label{ss:estimates}

We are now ready to compute and estimate the time derivative of the energy.
In a first moment, we aim at proving bounds for the paralinearized operator
$$
T_Lu\,:=\,\d_tu\,+\,i\,T_{\mc{A}}u\,+\,T_Bu\,.
$$
Notice that $T_Lu\,=\,Lu\,-\,\mc{R}_Lu$ and also $T_Lu\,=\,\wtilde{L}u\,-\,\mc{R}_{\wtilde L}u$.

In what follows, we will generically denote by $C_{LL}$ a multiplicative constant which depends on the $LL$ norms of
the coefficients of $\mc{A}(t,x,\xi)$ and of the symmetrizer $S(t,x,\xi)$, i.e. quantities $K_0$ and $K_1$ in \eqref{hyp:bound}-\eqref{hyp:LL},
but not on $s$ neither on $u$. On the other hand, we will use the generic symbol $C$ if the constant just depend on $K_0$,
i.e. the $L^\infty$ bounds of the coefficients and of the symmetrizer, and on $K_2$ in \eqref{hyp:Holder}, i.e. the  H\"older norms of $B$.
Finally, we will use the notation $C_p$ to denote a constant which depends on the $\mu$ fixed for having the positivity estimates of
Lemma \ref{l:energy}.

\begin{lemma} \label{l:est-T_L}
Let $\bigl(A_j\bigr)_{1\leq j\leq n}$ and $B$ be as in the hypotheses of Theorem \ref{th:en_LL}. For any $s\in\R$ and any $\beta>0$,
define $s(t)$ by formula \eqref{eq:s}. Then, for any smooth $u\in\mc{S}\bigl([0,T]\times\R^n\bigr)$ we have
\begin{eqnarray*}
\frac{d}{dt}E(t) & \leq & C_1\,E(t)\,+\, 
\left(C_2\,-\,C_3\,\beta\right)\,E_{\log}(t)\,+ \\
& & \;+\,\left\|\theta_\mu(D) T_Lu\right\|_{H^{s(t)}}\,\bigl(E(t)\bigr)^{1/2}\,+\,
2\left|\Re\!\left(\Lambda^{s(t)}(D)\,\wtilde{T}_{\Sigma}\,T_Lu\,,\,\Lambda^{s(t)}(D)\,\wtilde{T}_{\Sigma}u\right)_{L^2}\right|\,.
\end{eqnarray*}
The constant $C_1$ just depends on the $\mu$ fixed in the positivity estimates and on $K_0$; $C_2$ still depends on $\mu$ and $K_0$,
and also on $K_1$; finally, $C_3$ depends only on $K_0$.
\end{lemma}

\begin{proof}
We start by noticing that, for any $v\in\mc{S}(\R^n)$, we have
$$
\frac{d}{dt}\left\|v\right\|^2_{H^{s(t)}}\,=\,s'(t)\int_{\R^n}\log\!\left(1+|\xi|^2\right)\,\left(1+|\xi|^2\right)^{s(t)}\,
\left|\what{v}(\xi)\right|^2\,d\xi\;\sim\;s'(t)\,\left\|v\right\|^2_{H^{s(t)+(1/2)\log}}\,.
$$
For notation convenience, we set $\Lambda(D)\,:=\,(1-\Delta)^{1/2}$, i.e. $\Lambda(\xi)\,=\,\lan\xi\ran\,=\,\left(1+|\xi|^2\right)^{1/2}$.

Recalling definition \eqref{eq:E} of the energy $E$, we get:
\begin{eqnarray*}
& & \hspace{-0.5cm}
\frac{d}{dt}E(t)\,=\,s'(t)\,\Re\!\left(\log\!\left(\Lambda^2(D)\right)\,\Lambda^{s(t)}(D)\,\wtilde{T}_{\Sigma}u\,,\,
\Lambda^{s(t)}(D)\,\wtilde{T}_{\Sigma}u\right)_{L^2}\,+ \\
& & \quad +\,2\,\Re\!\left(\Lambda^{s(t)}(D)\,T_{\d_t\wtilde{\Sigma}}u\,,\,
\Lambda^{s(t)}(D)\,\wtilde{T}_{\Sigma}u\right)_{L^2}\,+\,2\,\Re\!\left(\Lambda^{s(t)}(D)\,\wtilde{T}_{\Sigma}\d_tu\,,\,
\Lambda^{s(t)}(D)\,\wtilde{T}_{\Sigma}u\right)_{L^2}\,+ \\
& & \quad +\,s'(t)\,\Re\!\left(\log\!\left(\Lambda^2(D)\right)\,\Lambda^{s(t)}(D)\,\theta_\mu(D)u\,,\,
\Lambda^{s(t)}(D)\,\theta_\mu(D)u\right)_{L^2}\,+ \\
& & \quad +\,2\,\Re\!\left(\Lambda^{s(t)}(D)\,\theta_\mu(D)\d_tu\,,\,
\Lambda^{s(t)}(D)\,\theta_\mu(D)u\right)_{L^2}\;=\;F_1\,+\,F_2\,+\,F_3\,+\,F_4\,+\,F_5\,.
\end{eqnarray*}

First of all, let us consider the terms with $s'(t)$: keeping in mind definitions \eqref{eq:E_log} and \eqref{eq:s}, it is easy to see
that, for some constant just depending on the $C_0$ appearing in Lemma \ref{l:energy},
\begin{equation} \label{est:s'}
F_1+F_4\,\leq\,-\,C\,\beta\,E_{\log}(t)\,.
\end{equation}
On the other hand, Theorem \ref{t:symb_time} implies
\begin{equation} \label{est:F_2}
\left|F_2\right|\,\leq\,\left\|T_{\d_t\wtilde{\Sigma}}u\right\|_{H^{s(t)-(1/2)\log}}
\left\|\wtilde{T}_{\Sigma}u\right\|_{H^{s(t)+(1/2)\log}}\,\leq\,C_{LL}\|u\|_{H^{s+(1/2)\log}}\,E^{1/2}_{\log}\,\leq\,
C_{LL}\,C_p\,E_{\log}\,.
\end{equation}

For both $F_3$ and $F_5$, we have to use the equation for $T_Lu$. We start by dealing with the low frequencies term:
\begin{eqnarray*}
F_5 & = & 2\,\Re\!\left(\Lambda^{s(t)}(D)\,\theta_\mu(D)\d_tu\,,\,
\Lambda^{s(t)}(D)\,\theta_\mu(D)u\right)_{L^2} \\
& = & 2\,\Re\!\left(\Lambda^{s(t)}(D)\,\theta_\mu(D)\left(T_Lu\,-\,i\,T_{\mc A}u\,-\,T_Bu\right)\,,\,
\Lambda^{s(t)}(D)\,\theta_\mu(D)u\right)_{L^2}\,.
\end{eqnarray*}
First of all, Cauchy-Schwarz inequality immediately implies
$$ 
\left|\Re\!\left(\Lambda^{s(t)}(D)\,\theta_\mu(D)\,T_Lu\,,\,
\Lambda^{s(t)}(D)\,\theta_\mu(D)u\right)_{L^2}\right|\,\leq\,\left\|\theta_\mu(D) T_Lu\right\|_{H^{s(t)}}\,E^{1/2}\,.
$$ 
As for the term with $T_{\mc A}$, we remark that, by spectral localization properties, one has
$$
\theta_\mu(D)\,T_{i\mc A}u\,=\,\sum_{j=1}^n\sum_{k=0}^{k_\mu}S_{k-3}\bigl(A_j(t,x)\bigr)\,\Delta_k\d_ju\,,
$$
for some $k_\mu\,\sim\,\log_2\mu$. Hence, Bernstein inequalities immediately imply
$$
\left|2\,\Re\!\left(\Lambda^{s(t)}(D)\,\theta_\mu(D)\,T_{i\mc A}u\,,\,
\Lambda^{s(t)}(D)\,\theta_\mu(D)u\right)_{L^2}\right|\,\leq\,C\,\|u\|_{H^{s(t)}}\,E^{1/2}\,\leq\,C\,C_p\,E(t)\,,
$$
for a suitable constant $C$ depending just on the $L^\infty$ norms of the $A_j$'s. Exactly in the same way, we get an analogous
estimate for the $T_B$ term. 
In the end, putting these inequalities together, we deduce the control
\begin{equation} \label{est:F_5}
\left|F_5\right|\,\leq\,C\,C_p\,E(t)\,+\,\left\|\theta_\mu(D) T_Lu\right\|_{H^{s(t)}}\,\bigl(E(t)\bigr)^{1/2}\,+\,
C_p\,C_{LL}\,E_{\log}(t)\,.
\end{equation}

Let us consider now the term $F_3$, which we rewrite as
$$
F_3\,=\,2\,\Re\!\left(\Lambda^{s(t)}(D)\,\wtilde{T}_{\Sigma}\left(T_Lu\,-\,i\,T_{\mc A}u\,-\,T_Bu\right)\,,\,
\Lambda^{s(t)}(D)\,\wtilde{T}_{\Sigma}u\right)_{L^2}\,.
$$
We leave the $T_L$ term on one side: it will contribute to the last item appearing in our statement. In addition, the term with $T_B$ is easy
to control: since $\wtilde{T}_{\Sigma}$ and $T_B$ are operators of order $0$, we get
$$ 
\left|\Re\!\left(\Lambda^{s(t)}(D)\,\wtilde{T}_{\Sigma}T_Bu\,,\,
\Lambda^{s(t)}(D)\,\wtilde{T}_{\Sigma}u\right)_{L^2}\right|\,\leq\,C\,\left\|u\right\|_{H^{s(t)}}\,E^{1/2}\,\leq\,C\,C_p\,E(t)\,.
$$ 
So, we have to focus just on the last term,
$$ 
\wtilde{F}_3\,:=\,2\,\Re\!\left(\Lambda^{s(t)}(D)\,\wtilde{T}_{\Sigma}\,T_{i\mc A}u\,,\,
\Lambda^{s(t)}(D)\,\wtilde{T}_{\Sigma}u\right)_{L^2}\,:
$$ 
we are going to make a systematic use of symbolic calculus, taking advantage of the properties established in Theorems \ref{t:symb_calc}
and \ref{t:symb_tilde}.

First of all, passing to the adjoints, we have the equality
$$
\wtilde{F}_3\,=\,2\,\Re\!\left(\wtilde{T}_{\Sigma}\,\Lambda^{2s(t)}(D)\,\wtilde{T}_{\Sigma}\,T_{i\mc A}u\,,\,u\right)_{L^2}\,+\,
2\,\Re\!\left(R_1\Lambda^{2s(t)}(D)\,\wtilde{T}_{\Sigma}\,T_{i\mc A}u\,,\,u\right)_{L^2}\,,
$$
where $R_1$ is of order $-1+\log$. On the other hand, by Remark \ref{r:p-prod} we have
$\wtilde{T}_{\Sigma}\,\Lambda^{2s(t)}(D)\,=\,\Lambda^{2s(t)}(D)\,\wtilde{T}_{\Sigma}\,+\,R_2$, where the operator
$R_2$ is a remainder of order $2s(t)-1+\log$.
Therefore, collecting the $R_1$ and $R_2$ terms into only one remainder $R(u,u)$, which can be estimated as
\begin{equation} \label{est:remainder}
\left|R(u,u)\right|\,\leq\,C_{LL}\,\|u\|_{H^{s(t)+(1/2)\log}}\,E^{1/2}_{\log}\,\leq\,C_p\,C_{LL}\,E_{\log}(t)\,,
\end{equation}
we can write
$$
\wtilde{F}_3\,=\,2\,\Re\!\left(\Lambda^{2s(t)}(D)\wtilde{T}_{\Sigma}\wtilde{T}_{\Sigma}T_{i\mc A}u,u\right)_{L^2}+R(u,u)\,=\,
2\,\Re\!\left(\Lambda^{2s(t)}(D)T_{\wtilde{\Sigma}^2}T_{i\mc A}u,u\right)_{L^2}+R(u,u)\,.
$$ 
Notice that, in the second step, we have included into $R(u,u)$ another rest, depending on a remainder
$R_3$ which can be still bounded as in \eqref{est:remainder}.

Now, we use the fact that, by definition, $\wtilde{\Sigma}^2\,=\,\wtilde{S}\,\bigl(1-\theta_\mu(\xi)\bigr)^2$. Since, by Theorem
\ref{t:symb_time}, the difference operator $S-\wtilde{S}$ contributes to the estimates as another remainder, in the sense
of inequality \eqref{est:remainder}, we arrive at the identity
$$
\wtilde{F}_3\,=\,2\,\Re\!\left(\Lambda^{s(t)}(D)\,T_{S\,(1-\theta_\mu)^2}\,T_{i\mc A}u\,,\,\Lambda^{s(t)}(D)\,u\right)_{L^2}\,+\,R(u,u)\,.
$$
Finally, by symbolic calculus again, up to adding another remainder $R_4$ of order $0+\log$ (recall that $\mc{A}$ is of order $1$), we get
$$
\wtilde{F}_3\,=\,2\,\Re\!\left(\Lambda^{s(t)}(D)\,T_{i\,S\,\mc{A}\,(1-\theta_\mu)^2}u\,,\,\Lambda^{s(t)}(D)\,u\right)_{L^2}\,+\,R(u,u)\,.
$$
At this point, we remark that $\Re\bigl(i\,S\,\mc{A}\bigr)\,=\,0$, since $S$ is a microlocal symmetrizer for $\mc{A}$: then,
keeping in mind that $\Re P\,=\,\bigl(P+P^*\bigr)/2$, by symbolic calculus we deduce that
$$
2\,\Re\left(T_{i\,S\,\mc{A}\,(1-\theta_\mu)^2}\right)\,=\,R_5\,,
$$
where also $R_5$ is an operator of order $0+\log$. In the end, putting all these informations together, we find the estimate
$$
\left|\wtilde{F}_3\right|\,\leq\,C_{LL}\,C_p\,E_{\log}(t)\,,
$$
which in turn gives us
\begin{equation} \label{est:F_3}
\left|F_3\right|\,\leq\,C\,C_p\,E(t)\,+\,C_p\,C_{LL}\,E_{\log}(t)\,+\,
2\left|\Re\!\left(\Lambda^{s(t)}(D)\,\wtilde{T}_{\Sigma}\,T_Lu\,,\,\Lambda^{s(t)}(D)\,\wtilde{T}_{\Sigma}u\right)_{L^2}\right|\,.
\end{equation}

Therefore, collecting inequalities \eqref{est:s'}, \eqref{est:F_2}, \eqref{est:F_5} and \eqref{est:F_3} completes the proof
of the energy estimates for the paralinearized operator $T_L$.
\end{proof}

\begin{rem} \label{r:s_values}
No restriction on $s$ is needed at this level: its limitations derive just from product rules (see Proposition
\ref{p:Hol-Sob}) and from the analysis of remainders (see Lemma \ref{l:L->T} above).
\end{rem}

We complete now the proof of Theorem \ref{th:en_LL} in the case of operator $L$. Operator $\wtilde{L}$ will be matter of Remark
\ref{r:adj} below.

It remains us to deal with the term $T_L$ in the estimates provided by Lemma \ref{l:est-T_L}.
Recall that $T_Lu\,=\,Lu\,-\,\mc{R}_Lu$.

For the low frequencies term, it is an easy matter to see that
\begin{eqnarray*}
\left\|\theta_\mu(D) T_Lu\right\|_{H^{s(t)}} & \leq & \left\|Lu\right\|_{H^{s(t)}}\,+\,
C_p\,\left\|\mc{R}_Lu\right\|_{H^{s(t)-(1/2)\log}} \\
& \leq & \left\|Lu\right\|_{H^{s(t)}}\,+\,C_p\,\left(C_{LL}+C\right)\,
\left\|u\right\|_{H^{s(t)+(1/2)\log}}
\end{eqnarray*}
where $C_p$ is, as usual, a constant which depends on the positivity estimates. The presence of the constant $C$ in the last
step is due to lower order terms, i.e. $\mc{R}_B$.
Applying once more positivity estimates, we finally arrive to the bound
\begin{equation} \label{est:low-freq}
\left\|\theta_\mu(D) T_Lu\right\|_{H^{s(t)}}\,\leq\,\left\|Lu\right\|_{H^{s(t)}}\,+\,C_p\,\left(C_{LL}+C\right)\,E_{\log}(t)\,.
\end{equation}

We remark here that, up to let $C_p$ depend
also on $s$, the previous inequality holds true for any $s\in\R$: we have no need yet for using the condition $0<s<\g$.

As for the high frequencies term, since $\wtilde{T}_\Sigma$ is an operator of order $0$, we have the control
\begin{eqnarray*}
& & \hspace{-0.2cm}
2\left|\Re\!\left(\Lambda^{s(t)}(D)\,\wtilde{T}_{\Sigma}\,T_Lu\,,\,\Lambda^{s(t)}(D)\,\wtilde{T}_{\Sigma}u\right)_{L^2}\right|\,\leq\,
C\,\left\|Lu\right\|_{H^{s(t)}}\,E^{1/2}\,+\,C\,\left\|\mc{R}_Lu\right\|_{H^{s(t)-(1/2)\log}}\,E_{\log}^{1/2} \\
& & \qquad\qquad\qquad\qquad\qquad
\leq\,C\,\left\|Lu\right\|_{H^{s(t)}}\,E^{1/2}\,+\,C\left(C_{LL}+C\right)\left\|u\right\|_{H^{s(t)+(1/2)\log}}\,E_{\log}^{1/2}\,,
\end{eqnarray*}
where we have used also Lemma \ref{l:L->T}, and hence the restriction on $s$. Therefore, by positivity estimates again we deduce
\begin{equation} \label{est:high-freq}
2\left|\Re\!\left(\Lambda^{s(t)}(D)\,\wtilde{T}_{\Sigma}\,T_Lu\,,\,\Lambda^{s(t)}(D)\,\wtilde{T}_{\Sigma}u\right)_{L^2}\right|\,\leq\,
C\,\left\|Lu\right\|_{H^{s(t)}}\,E^{1/2}\,+\,C_p\left(C_{LL}+C\right)\,E_{\log}(t)\,.
\end{equation}

Now, putting \eqref{est:low-freq} and \eqref{est:high-freq} into the inequality given by Lemma \ref{l:est-T_L}, we find
$$
\frac{d}{dt}E(t)\,\leq\,C_1\,E(t)\,+\,C_3\,\left\|Lu\right\|_{H^{s(t)}}\,\bigl(E(t)\bigr)^{1/2}\,+\,\bigl(C_4-\beta\bigr)\,E_{\log}(t)\,,
$$
where $C_1$ is the constant given by Lemma \ref{l:est-T_L}, $C_3>0$ depends on $K_0$ and $K_1$, but just via positivity estimates,
and $C_4$ depends on $K_0$, $K_1$ and $K_2$ not only via Lemma \ref{l:energy}, but also via Lemma \ref{l:L->T}.

Now, we chose $\beta>C_{4}$: then, setting $e(t)\,:=\,\bigl(E(t)\bigr)^{1/2}$, an application of Gronwall inequality leads us to
the estimate
$$
e(t)\,\leq\,M\,e^{Qt}\left(e(0)\,+\,\int^t_0\left\|Lu(\tau)\right\|_{H^{s(\tau)}}\,d\tau\right)\,,
$$
for two positive constants $M$ and $Q$ large enough.
This completes the proof of the energy estimate stated in Theorem \ref{th:en_LL}, for operator $L$.

\begin{rem} \label{r:adj}
For operator $\wtilde{L}$ one can argue in a completely analogous way. The only difference is the presence of the remainder operator
${\mc R}_{\wtilde L}$ instead of $\mc{R}_L$, as stated in Lemma \ref{l:L->T}. However, since the order of the two operators is the same,
it is easy to see that the estimates do not change.
\end{rem}

\begin{rem} \label{r:precise-est}
As already remarked in \cite{C-M}, we point out that, in fact, our proof gives a more accurate energy estimate: namely,
\begin{eqnarray}
& & \hspace{-1cm}
\sup_{t\in[0,T_*]}\|u(t)\|_{H^{s-\beta t}}\,+\,\left(\int^{T_*}_0\|u(\tau)\|^2_{H^{s-\beta\tau+(1/2)\log}}\,d\tau\right)^{1/2}\,\leq \label{est:precise-LL} \\
& & \qquad\qquad\qquad\qquad\qquad\qquad
\leq\,C_1\,e^{C_2\,T}\,\left(\|u(0)\|_{H^s}\,+\,\int^{T_*}_0
\bigl\|Lu(\tau)\bigr\|_{H^{s-\beta\tau}}\,d\tau\right) \nonumber
\end{eqnarray}
for tempered distributions $u$ as in the statement of Theorem \ref{th:en_LL}. Moreover, if the last term is $L^2$ in time, one can replace it by
$\int_0^{T_*}\bigl\|Lu(\tau)\bigr\|^2_{H^{s-\beta\tau-(1/2)\log}}\,d\tau$.
\end{rem}

\begin{rem} \label{r:t-LL}
A careful inspection of our proof reveals that we just used the property $A_j\,\in\,L^\infty\bigl([0,T];LL(\R^n;\mc{M}_m)\bigr)$,
namely only $LL$ regularity in $x$ is exploited for the $A_j$'s. On the contrary, for the symmetrizer $S$ one needs log-Lipschitz continuity \emph{both} in time and space variables.

Nonetheless, in general regularity of the symmetrizer is dictated by the regularity of the coefficients (keep in mind the discussion in Example \ref{ex:const}). Furthermore,
hypothesis \eqref{hyp:LL} is invariant by change of coordinates, and then suitable for local analysis. This is why we required it.
\end{rem}

\subsection{Existence and uniqueness of solutions} \label{ss:global_e}

In this subsection, we prove Theorem \ref{t:global_e}, namely the existence and uniqueness of solutions to the global Cauchy problem \eqref{eq:Cauchy}.
For this, we will exploit in a fundamental way the energy estimates of Theorem \ref{th:en_LL}.

We will focus on the case of operator $L$, defined in \eqref{def:L}. The same arguments hand over $\wtilde{L}$, see its definition in \eqref{def:L*}, rather directly.

\subsubsection{Regularity results, uniqueness} \label{sss:global_reg}

As done in \cite{C-M} (see Definition 2.5 of that paper), for $(\sigma,\alpha)\in\R^2$ and $\beta>0$ fixed, let us define the following spaces:
\begin{itemize}
 \item $\mc{C}_{\sigma+\alpha\log,\beta}(T)$ is the set of functions $v$ on $[0,T]$, with values in the space of tempered distributions, which verify the property
$$
u\in\mc{C}\bigl([0,t];H^{\sigma-\beta t+\alpha\log}\bigr)
$$
for all $t\in[0,T]$;
\item $\mc{H}_{\sigma+\alpha\log,\beta}(T)$ is the set of functions $w$ on $[0,T]$, with values in the space of tempered distributions, such that
$$
\Lambda^{\sigma-\beta t}(D)\,\Pi^{\alpha}(D)\,w(t)\;\in\;L^2\bigl([0,T];L^2(\R^n;\R^m)\bigr)\,,
$$
where we used the notations $\Lambda(D)\,=\,(1-\Delta)^{1/2}$ and $\Pi(D)\,=\,\log(2+|D|)$, introduced in the previous paragraphs;
\item $\mc{L}_{\sigma+\alpha\log,\beta}(T)$ is the set of functions $z$ on $[0,T]$, with values in the space of tempered distributions, with the property
$$
\Lambda^{\sigma-\beta t}(D)\,\Pi^{\alpha}(D)\,z(t)\;\in\;L^1\bigl([0,T];L^2(\R^n;\R^m)\bigr)\,.
$$
\end{itemize}
All these spaces are endowed with the natural norms induced by the conditions here above.

\medbreak
To begin with, we deal with \emph{weak solutions} to the Cauchy problem \eqref{eq:Cauchy}. More precisely, fix $\g\in\,]0,1[\,$ and $s\in\,]0,\g[\,$
as in the statement of Theorem \ref{t:global_e}, and take the positive constants $\beta$ and $T_*$ given by Theorem \ref{th:en_LL}.
For $u_0\in H^s$ and $f\in\mc{L}_{s,\beta}(T_*)$, we are going to consider
$$
u\,\in\,\mc{H}_{s,\beta}(T_*)
$$
such that the equation $Lu=f$ is satisfied, together with the initial condition $u_{|t=0}=u_0$, in the sense of distributions.

We remark that, for such $u$, the product $A_j(t,x)\,\d_ju$ is well-defined and belongs to the space $\mc{H}_{s-1,\beta}(T_*)$ for all $j$ (recall Corollary \ref{c:LL-H^s}), and
$B(t,x)u$ belongs to $\mc{H}_{s,\beta}(T_*)$ (see Proposition \ref{p:Hol-Sob}). Hence, the equation $Lu=f$ makes sense in $\mc{D}'\bigl(\,]0,T_*[\,\times\R^n;\R^m\bigr)$.

Let us show now that it makes sense also to impose the initial condition $u_{|t=0}=u_0$.
\begin{lemma} \label{p:reg_global}
Let $\g$ and $s$ be as above. Let $u\in\mc{H}_{s,\beta}(T_*)$ verify the equation $Lu=f$ in $\mc{D}'$
for some $f\in\mc{L}_{s,\beta}(T_*)$.

Then one has $u\in\mc{C}_{s-(1/2),\beta}(T_*)$. In particular, the trace $u_{|t=0}$ is well-defined in $H^{s-1/2}(\R^n)$,
and the initial condition in \eqref{eq:Cauchy} makes sense.
\end{lemma}

\begin{proof}
The proof is somehow classical. Let us define $v(t)\,:=\,u\,-\,\int^t_0f(\tau)d\tau$. Then, by hypotheses it easily follows that $\int^t_0f\,\in\,\mc{C}_{s,\beta}(T_*)$,
and hence $v\,\in\,\mc{H}_{s,\beta}(T_*)$.

On the other hand, from the arguments exposed before, we deduce that $\d_tv\,=\,\sum_{j}A_j\d_ju\,+\,Bu$ belongs to the space $\mc{H}_{s-1,\beta}$. Then, an easy
interpolation implies that $v\in\mc{C}_{s-(1/2),\beta}$, and therefore so does $u$.
\end{proof}

As a result of this lemma and the previous considerations, we have clarified the sense to give to the Cauchy problem \eqref{eq:Cauchy}.

Next, let us state a ``weak $=$ strong'' type result. Namely, we show that any weak solution $u$ is in fact the limit of a suitable sequence of smooth approximate solutions,
in the norm given by the left-hand side of inequality \eqref{est:u_LL}. In particular, this fact implies that $u$ enjoys additional smoothness and it
satisfies the energy estimates stated in Theorem \ref{th:en_LL}.
\begin{thm} \label{th:w-s}
Fix $\g\,\in\,]0,1[\,$ and $0<s<\g$. Let $\beta>0$ and $T_*$ be the ``loss parameter'' and the existence time given by Theorem \ref{th:en_LL}, together with the constants
$C_1$ and $C_2$. \\
Let $u_0\in H^s(\R^n;\R^m)$, $f\in\mc{L}_{s,\beta}(T_*)$ and $u\in\mc{H}_{s,\beta}(T_*)$ be the corresponding weak solution to the Cauchy problem
\eqref{eq:Cauchy}.

Then one has $u\,\in\,\mc{C}_{s,\beta}(T_*)$, and $u$ satisfies the energy inequality \eqref{est:u_LL}.
\end{thm}

\begin{proof}
For $\veps\in\,]0,1]$, let us introduce the smoothing operators $J_\veps(D)\,:=\,(1-\veps\Delta)^{-1/2}$, which are Fourier multipliers associated to the symbols
$J_\veps(\xi)\,=\,\left(1+\veps|\xi|^2\right)^{-1/2}$.

Recall that they are operators of order $-1$, they are uniformly bounded from $H^\sigma$ to $H^{\sigma+1}$ for any $\sigma\in\R$, and one has the strong convergence
$J_\veps(D)w\,\longra\,w$ in $H^\sigma$, for $\veps\ra0$.

For any $\veps>0$, we set $u_\veps\,:=\,J_\veps(D)u$. Then, from the previous properties we deduce that the family $\bigl(u_\veps\bigr)_\veps$ is uniformly bounded in
$\mc{H}_{s+1,\beta}(T_*)\,\hra\,L^2\bigl([0,T_*];H^1(\R^n;\R^m)\bigr)$. Moreover, we have that $u_\veps\,\longra\,u$ in $\mc{H}_{s,\beta}(T_*)$ for $\veps\ra0$.

In addition, if we apply operator $J_\veps(D)$ to the equation $Lu=f$, we deduce that the $u_\veps$'s solve the equation $Lu_\veps\,=\,f_\veps+g_\veps$, with
$f_\veps=J_\veps(D)f$ and
\begin{equation} \label{def:g_eps}
g_\veps\,:=\,\sum_{j=1}^n\bigl[A_j(t,x),J_\veps(D)\bigr]\d_ju\,+\,\bigl[B(t,x),J_\veps(D)\bigr]u\,.
\end{equation}
By Lemma 4.6 of \cite{C-M}, it is easy to infer that $f_\veps\,\longra\,f$ in $\mc{L}_{s,\beta}(T_*)$ and $g_\veps\,\longra\,0$ in $\mc{H}_{s,\beta}(T_*)$,
in the limit for $\veps\ra0$.

From this fact, Theorem \ref{th:en_LL} and Remark \ref{r:precise-est} together, we deduce the energy estimates
\begin{eqnarray*}
& & \hspace{-0.7cm}
\sup_{t\in[0,T_*]}\|u_\veps(t)\|_{H^{s-\beta t}}\,+\,\left(\int^{T_*}_0\|u_\veps(\tau)\|^2_{H^{s-\beta\tau+(1/2)\log}}\,d\tau\right)^{1/2}\,\leq \\
& & \qquad\qquad\qquad\qquad\qquad
\leq\,C_1\,e^{C_2\,T}\,\left(\|u_\veps(0)\|_{H^s}\,+\,\int^T_0\left(\bigl\|f_\veps\bigr\|_{H^{s-\beta\tau}}\,+\,\bigl\|g_\veps\bigr\|_{H^{s-\beta\tau}}\right)\,d\tau\right)\,.
\end{eqnarray*}
By linearity of the equations, similar estimates are satisfied also by the difference $\delta u_{\veps,\eta}\,:=\,u_\veps-u_\eta$, for all (say) $0<\eta<\veps\leq1$,
up to replace, in the right-hand side, $u_\veps(0)$, $f_\veps$ and $g_\veps$ respectively by $\delta u_{\veps,\eta}(0)\,=\,u_\veps(0)-u_\eta(0)$, $\delta f_{\veps,\eta}=f_\veps-f_\eta$
and $\delta g_{\veps,\eta}=g_\veps-g_\eta$.

Since $\delta u_{\veps,\eta}(0)\longra0$ in $H^s$, and so do $\delta f_{\veps,\eta}$ and $\delta g_{\veps,\eta}$ respectively in $\mc{L}_{s,\beta}(T_*)$ and in
$\mc{H}_{s,\beta}(T_*)$, we gather that $\bigl(u_\veps\bigr)_\veps$ is a Cauchy sequence in $\mc{C}_{s,\beta}(T_*)$. Therefore, the limit $u\in\mc{H}_{s,\beta}(T_*)$
also belongs to $\mc{C}_{s,\beta}(T_*)$, and it verifies the energy estimate \eqref{est:u_LL}.
\end{proof}

The previous theorem immediately implies uniqueness of weak solutions.
\begin{coroll} \label{c:w_uniq}
Let $\g$ and $s$ be fixed as in the hypotheses of Theorem \ref{th:w-s}. Let $u\in\mc{H}_{s,\beta}(T_*)$ be a weak solution to the Cauchy problem \eqref{eq:Cauchy},
with initial datum $u_0=0$ and external force $f\equiv0$.

Then $u\,\equiv\,0$.
\end{coroll}

\subsubsection{Existence of weak solutions} \label{sss:weak}

In order to complete the proof of Theorem \ref{t:global_e}, it remains us to prove existence of weak solutions.
\begin{prop} \label{p:w_existence}
Fix $\g\,\in\,]0,1[\,$ and $0<s<\g$. Let $\beta>0$ and $T_*$ be the ``loss parameter'' and the existence time given by Theorem \ref{th:en_LL}.

Then, for all $u_0\in H^s(\R^n;\R^m)$ and $f\in\mc{L}_{s,\beta}(T_*)$, there exists a weak solution $u\in\mc{H}_{s,\beta}(T_*)$ to the Cauchy problem
\eqref{eq:Cauchy}.
\end{prop}

\begin{proof}
For $\veps\in\,]0,1]$, let us introduce the smoothing operators $J_\veps(D)$, of order $-1$, as done in the previous proof.
Denoting $A(t,x,D)\,=\,\sum_jA_j(t,x)\d_j$, for any $\veps$ let us consider the linear system of $m$ ODEs
$$
\d_tu_\veps\,=\,-\,A(t,x,D)\,J_\veps(D)\,u_\veps\,-\,B(t,x)\,J_\veps(D)\,u_\veps\,+\,f\,,
$$
with initial datum $(u_\veps)_{|t=0}\,=\,u_0$.

For all $\veps$ fixed and all $t$, it is easy to see that the operators $A(t,x,D)J_\veps(D)$ and $B(t,x)J_\veps(D)$
are bounded in $L^2$. Moreover, by our hypotheses we have $s>s-\beta T_*>0$, so that we can solve the previous system in $\mc{C}\bigl([0,T_*];L^2(\R^n;\R^m)\bigr)$
by use of the Cauchy-Lipschitz Theorem, and find a solution $u_\veps$.
On the other hand, the hypothesis over $f$ implies that, for any $t_0\in[0,T_*]$ fixed, $f\in L^1\bigl([0,t_0]; H^{s-\beta t_0}\bigr)$. Furthermore, by product rules
(see Proposition \ref{p:Hol-Sob} and Corollary \ref{c:LL-H^s}), the operators $A(t,x,D)J_\veps(D)$ and $B(t,x)J_\veps(D)$ are self-maps of $H^{s-\beta t_0}$ into itself.
Therefore, by Cauchy-Lipschitz Theorem again and uniqueness part, we infer that $u\in\mc{C}_{s,\beta}(T_*)\,\hra\,\mc{H}_{s,\beta}(T_*)$. 

Notice however that this is \emph{not} enough to get uniform bounds on the family $(u_\veps)_\veps$, since we do not know if we have enough regularity in order to absorbe, in energy estimates,
the remainders which require an additional $(1/2)\log$-regularity (see the computations in Subsection \ref{ss:estimates} above). Hence, we are going to argue in a slightly different way.

Let us define $w_\veps\,:=\,J_\veps(D)u_\veps$: by the previous argument, $w_\veps$ is in $L^2\bigl([0,T_*];H^1(\R^n;\R^m)\bigr)$ for all $\veps\in\,]0,1]$.
Moreover, it satisfies
\begin{eqnarray*}
\d_tw_\veps & = & -\,J_\veps(D)\,A(t,x,D)\,w_\veps\,-\,J_\veps(D)\,B(t,x)\,w_\veps\,+\,J_\veps(D)\,f \\
& = & -\,A(t,x,D)J_\veps(D)\,w_\veps\,-\,B(t,x)J_\veps(D)\,w_\veps\,+\,f_\veps\,+\,h_\veps\,,
\end{eqnarray*}
where $f_\veps\,=\,J_\veps(D)f$ as before, and $h_\veps$ is defined by the analogue of formula \eqref{def:g_eps}, but replacing $u$ by $w_\veps$ itself.
Notice that Lemma 4.6 of \cite{C-M} implies the inequality
$$
\left\|h_\veps\right\|_{H^{s(t)}}\,\leq\,C\,\|w_\veps\|_{H^{s(t)}}\,,
$$
where $s(t)\,=\,s-\beta t$ as above. Let us also remark that $(w_\veps)_{|t=0}\,=\,\Lambda_\veps(D)u_0$.

Then, we can apply energy estimates of Theorem \ref{th:en_LL} to $w_\veps$.
Indeed, $J_\veps(\xi)$ being a scalar multiplier, $S(t,x,\xi)$ is still a microlocal symmetrizer for $A(t,x,\xi)J_\veps(\xi)$.
Moreover, Lemma \ref{l:L->T} gives uniform bounds for the remainder operators $A(t,x,\xi)J_\veps(\xi)-T_{A(t,x,\xi)J_\veps(\xi)}$ and
$B(t,x)J_\veps(\xi)-T_{B(t,x)J_\veps(\xi)}$ in suitable functional spaces. Finally, Lemma \ref{l:est-T_L} provides with uniform bounds for the operators
$\d_t+T_{iA(t,x,\xi)J_\veps(\xi)}+T_{B(t,x)J_\veps(\xi)}$, since the symbols are uniformly bounded respectively in the classes $\Gamma^{1+0\log}_{LL}$
and $\Gamma^{0+0\log}_{\g+0\log}$.

So we find that $\bigl(w_\veps\bigr)_\veps$ is a bounded family in $\mc{C}_{s,\beta}(T_*)\,\hra\,\mc{H}_{s,\beta}(T_*)$, and consequently, up to extraction of a subsequence,
it weakly converges to some $u$ in this space. On the other hand, by the equation for $w_\veps$ (recall also product rules of Proposition \ref{p:Hol-Sob} and Corollary \ref{c:LL-H^s}),
we easily deduce that $\bigl(\d_tw_\veps\bigr)_\veps$ is bounded in $\mc{C}_{s-1,\beta}(T_*)$, and hence the convergence holds true also in
the weak-$*$ topology of $H^1\bigl([0,T_*];H^{s-\beta T_*-1}(\R^n;\R^m)\bigr)\,\hra\,\mc{C}\bigl([0,T_*];H^{s-\beta T_*-1}\bigr)$.
In particular, 
this implies that $(w_\veps)_{|t=0}\,\longra\,u_{|t=0}$ in the distributional sense. 

Thanks to these properties, it is easy to pass to the limit in the weak formulation of the equations, 
obtaining thus that $u$ solves the system $Lu=f$ in a weak sense, with initial datum $u_0$.
Finally, by uniform bounds we get $u\,\in\,\mc{H}_{s,\beta}(T_*)$, and this fact completes the proof of the existence of a weak solution. 
\end{proof}

\section{The local Cauchy problem} \label{s:local}

We prove here local in space existence and uniqueness of solutions. First of all, let us show that it makes sense
to consider the Cauchy problem for operator $P$ on $\varSigma$: this is not clear \textit{a priori}, due to the low regularity framework.

Here below we will use the notations introduced in Subsection \ref{ss:local_th}. In particular, recall that we have set $\Omega_\geq:=\Omega\,\cap\,\{\vphi\geq0\}$ and
$\Omega_>:=\Omega\,\cap\,\{\vphi>0\}$, where $\{\vphi=0\}$ is a parametrization of the hypersurface $\varSigma$ in $\Omega$.

Recall that we have supposed that hypotheses from \textbf{(H-1)} to \textbf{(H-5)} (stated in Subsection \ref{ss:local_th}) hold true. In particular,
$P$ has log-Lipschitz first order coefficients and a $\g$-H\"older continuous $0$-th order coefficient, and it admits a family of full symmetrizers, which are
smooth in the dual variable and log-Lipschitz in the $z$ variable.

\subsection{Giving sense to the local Cauchy problem} \label{ss:sense}

We start by noticing that, for smooth $u$ and $v$, with ${\rm supp}\,v$ compact in $\Omega_\geq$, we can write the next identity only formally, due to the low regularity of the coefficients:
$$ 
\bigl(Pu\,,\,v\bigr)_{L^2(\Omega_>)}\,-\,\bigl(u\,,\,P^*v\bigr)_{L^2(\Omega_>)}\,=\,
\bigl(D_\varSigma u\,,\,N_{\nu,P_1} v\bigr)_{L^2(\varSigma)}\,,
$$ 
where $P^*(z,\z)v\,=\,-\sum_j\d_{z_j}\!\left(A_j^*v\right)\,+\,B^*v$ is the adjoint operator of $P$, and we  have defined
$$
D_\varSigma u\,:=\,u_{|\varSigma}\qquad\mbox{ and }\qquad N_{\nu,P_1}v\,:=\,\sum_{j=0}^n\nu^j\,D_\varSigma\!\left(A_j^*v\right)\,.
$$
As above, $\nu$ is the normal to the hypersurface $\varSigma$, which determines the integration form on $\varSigma$.
Of course, the map $z\,\mapsto\,\nu(z)\neq0$ is smooth on $\varSigma$.

In our framework the previous Green formula does not hold true a priori, due to the low regularity of the coefficients. The first step of the proof to Theorem \ref{t:local_e} is to justify
it for smooth enough functions.
To begin with, let us study the regularity of the terms entering in the definition of $Pu$ and $P^*u$: we have the following lemma.
\begin{lemma} \label{l:reg_P}
\begin{itemize}
 \item[(i)] For $s\in\,]1-\g,1+\g[\,$ and $u\in H^s_{loc}(\Omega_\geq)$, all the terms entering in the definition of $Pu$
 are well-defined in $H^{s-1}_{loc}(\Omega_\geq)$.
\item[(ii)] For $\sigma\in\,]1-\g,1[\,$ and $v\in H^\sigma_{loc}(\Omega_\geq)$, all the terms entering in the definition
of $P^*v$ are well-defined in $H^{\sigma-1}_{loc}(\Omega_\geq)$.
\item[(iii)] For $s>1/2$ and $u\in H^s_{loc}(\Omega_\geq)$, the trace $D_\varSigma u$ is well-defined in $H^{s-1/2}_{loc}(\Sigma\cap\Omega)$.
\item[(iv)] For $\sigma\in\,]1/2,1[\,$ and $v\in H^{\sigma}_{loc}(\Omega_\geq)$, the normal trace
$N_{\nu,P_1}v$ is well-defined in $H^{s-1/2}_{loc}(\varSigma\cap\Omega)$.
\end{itemize}
\end{lemma}

\begin{proof}
We repeatedly use product rules stated in Proposition \ref{p:Hol-Sob} and Corollary \ref{c:LL-H^s} above.

In a first time, let us focus on $P$: for all $s\in\,]0,2[\,$, $P_1(z,\d_z)u$ belongs to $H^{s-1}_{loc}(\Omega_\geq)$. Now,  by the
embedding $H^s\hookrightarrow H^{s-1}$, if $|s-1|<\g$ we have that $Bu\in H^{s-1}_{loc}(\Omega_\geq)$. This completes the proof of
point \textit{(i)}.

Concerning $P^*$, the argument is analogous: the principal part of $P^*u$ belongs to $H^{\sigma-1}_{loc}(\Omega_\geq)$ whenever
$\sigma\in\,]0,1[\,$. The $B^*u$ term can be treated exactly as before. Also \textit{(ii)} is proved.

Points \textit{(iii)} and \textit{(iv)} are straightforward.
\end{proof}

Following the discussion in \cite{C-M}, we remark that $\mc{C}^\infty_0(\Omega_>)$ is a dense subset of $H^\sigma(\Omega_>)$ for $|\sigma|<1/2$,
and that, when $\sigma\in[0,1/2[\,$ and $u\in H^\sigma(\Omega_>)$, the pairing $\bigl(u\,,\,v\bigr)_{L^2(\Omega_>)}$ for $v\in L^2$ extends
to the duality $H^\sigma\times H^{-\sigma}$. From this and Lemma \ref{l:reg_P}, we deduce the next statement, which tells us that Green formula makes sense for regular enough distributions.
\begin{lemma} \label{l:duality}
 Let us fix $1/2<\g<1$, and take $s\in\,]1/2,\g[\,$. For any $u\in H^s_{loc}(\Omega_\geq)$ and $v\in H^s_{comp}(\Omega_\geq)$,
 one has the equality
\begin{equation} \label{eq:by-parts}
\bigl(Pu\,,\,v\bigr)_{H^{-\sigma}\times H^\sigma}\,-\,\bigl(u\,,\,P^*v\bigr)_{H^{\sigma}\times H^{-\sigma}}\,=\,
\bigl(D_\varSigma u\,,\,N_{\nu,P_1} v\bigr)_{L^2(\varSigma)}\,,
\end{equation}
where $\sigma\,=\,1-s\,\in\;]1-\g,1/2[$ (and in particular $0\leq\sigma<1/2$).
\end{lemma}

\begin{proof}
Since $s>1/2$, we have $H^s\hookrightarrow H^{1-s}$. On the other hand, $s$ belongs in particular to $\,]1-\g,1[\,$,
and then Lemma \ref{l:reg_P} applies.

To complete the proof, it is enough to remark that (see Lemma 1.3 of \cite{C-M}) the Green's formula
$$
\bigl(\d_ju\,,\,v\bigr)_{H^{-\sigma}\times H^\sigma}\,=\,-\,\bigl(u\,,\,\d_jv\bigr)_{H^{\sigma}\times H^{-\sigma}}\,+\,
\bigl(\nu^j\,D_\varSigma u\,,\,D_\varSigma v\bigr)_{L^2(\varSigma)}\,,
$$
holds true for any $u\in H^{1-\sigma}_{loc}(\Omega_\geq)$ and $v\in H^{1-\sigma}_{comp}(\Omega_\geq)$, whenever $\sigma\in[0,1/2[\,$.
\end{proof}

As the final step, we want to justify Green formula for distributions which are in the domain of our operator. This is guaranteed by the next statement, which is the analogue to
Proposition 1.4 of \cite{C-M}. Thanks to this result, considering the Cauchy problem $(C\!P)$ under the hypotheses of Theorem \ref{t:local_e} makes sense.
\begin{prop} \label{p:extension}
Let us define the set $\mc{D}(P;H^s)\,:=\,\left\{u\in H^s_{loc}(\Omega_\geq)\,\bigl|\,Pu\in H^s_{loc}(\Omega_\geq)\right\}$.

There exists a unique extension of the operator $D_\varSigma$ to the set $\mc{D}(P)\,:=\,\bigcup_{s>1-\g}\mc{D}(P;H^s)$, which acts continuously
from $\mc{D}(P;H^s)$ into $H^{s-1/2}_{loc}(\varSigma\cap\Omega)$ for all $s\in\,]1-\g,\g[\,$.

The same property holds true for the operator $N_{\nu,P_1}$.

Furthermore, for all $s_0\in\,]1-\g,1/2[\,$ such that $s_0\leq s$, for all $v\in H^{1-s_0}_{comp}(\Omega_\geq)$, one has
the Green formula
$$
\bigl(Pu\,,\,v\bigr)_{L^2}\,-\,\bigl(u\,,\,P^*v\bigr)_{H^{s_0}\times H^{-s_0}}\,=\,
\bigl(D_\varSigma u\,,\,N_{\nu,P_1} v\bigr)_{H^{s-1/2}\times H^{1/2-s}}\,.
$$
\end{prop}

This proposition will be proved in the next subsection, and it will be derived by results on the global Cauchy problem. Then, in order to give sense to the Cauchy problem,
the microlocally symmetrizability and hyperbolicity hypotheses are fundamental.

Let us notice here that, thanks to definitions, the properties for $D_\varSigma$ easily pass also on $N_{\nu,P_1}$.
We observe also that all the terms entering in the last formula have sense. Indeed, $1/2<1-s_0<\g$ and hence, by Lemma \ref{l:reg_P},
we get $P^*v\in H^{-s_0}_{comp}(\Omega_\geq)$, so that the pairing with $u$ makes sense (recall that $s_0\leq s$).
Moreover one has $N_{\nu,P_1}v\in H^{1/2-s_0}_{comp}(\Omega_\geq)\hookrightarrow H^{1/2-s}_{comp}(\Omega_\geq)$,
and so also the last term in the equality is fine.

\subsection{Invariance by change of variables, and regularity results} \label{ss:change}

First of all, let us show here that our working hypotheses are invariant under smooth change of coordinates.

Indeed, let $z=\psi(y)$, for some $\psi$ smooth, and denote $\wtilde{f}(y)=f\circ\psi(y)$: derivatives change according
to the rule $\nabla_{z}u\,\circ\,\psi\,=\,^t\!\left(\nabla_z\psi^{-1}\right)\circ\psi\,\cdot\,\nabla_{y}\wtilde{u}$, namely 
$$
\left(\d_{z_j}u^k\right)\circ\psi\,=\,\sum_{h=1}^m\d_{z_j}\!\left(\psi^{-1}\right)^h\circ\psi\;\d_{y_h}\wtilde{u}^k\,,
$$
as well as covectors in the cotangent space, i.e. $\wtilde{\z}(y)\,=\,\left(\,^t\nabla_z\psi^{-1}\cdot \z\right)\circ\psi(y)$.
As a consequence, if we write the $i$-th component of our system $Pu=f$, that is to say 
$$
f^i\,=\,\sum_{j=0}^n\sum_{k=1}^mA_{j,ik}(z)\,\d_{z_j}u^k(z)\,+\,\sum_{k=1}^m B_{ik}(z)\,u^k(z)\,,
$$
from the previous rules we obtain the expression in $y$ coordinates:
$$
\wtilde{f}^i\,=\,\sum_{h=0}^n\sum_{k=1}^m M_{h,ik}(y)\,\d_{y_j}\wtilde{u}^k(y)\,+\,\sum_{k=1}^m \wtilde{B}_{ik}(y)\,\wtilde{u}^k(y)\,,
$$
where easy computations lead to the formula
$$
M_{h,ik}(y)\,=\,\sum_{j=0}^n\wtilde{A}_{j,ik}(y)\,\d_{z_j}\!\left(\psi^{-1}\right)^h\bigl(\psi(y)\bigr)\,.
$$

Namely, we have proved that $\wtilde{Pu}\,=\,\wtilde{P}\wtilde{u}$, where we have defined the operator
$$
\wtilde{P}v\,=\,\sum_{h=0}^nM_h\,\d_{y_h}v\,+\,\wtilde{B}v\,.
$$
Notice that $\wtilde{P}$ has the same form as $P$, and the same regularity of its first and $0$-th order coefficients.
In addition, by Theorem 4.11 of \cite{M-2014} we deduce that it is still microlocally symmetrizable in the sense of Definition \ref{d:full-symm}.
We can check this property also by direct computations: as a matter of fact, in a very natural way, let us define
$$
\wtilde{\mbf{S}}(y,\wtilde{\z})\,:=\,\mbf{S}\bigl(\psi(y)\,,\,\left(\,^t\nabla_z\psi^{-1}\cdot\z\right)\circ\psi\bigr)
$$
Of course, all the properties of $\mbf{S}$ hand over to $\wtilde{\mbf{S}}$; the only thing we have to check is that
$\wtilde{\mbf{S}}\wtilde{P}_1$ is self-adjoint, where we denote by $\wtilde{P}_1$ the principal part of $\wtilde{P}$. But this imediately
follows from the definitions and the transformation rules:
\begin{eqnarray*}
\left(\wtilde{\mbf{S}}\wtilde{P}_1\right)(y,\wtilde\z) & = & \wtilde{\mbf{S}}(y,\wtilde\z)\sum_{h=0}^n\wtilde{\z}^h\,M_h\;=\;
\wtilde{\mbf{S}}(y,\wtilde\z)\sum_{h=0}^n\sum_{j=0}^n\wtilde{A}_j\,\d_{z_j}\!\left(\psi^{-1}\right)^h\,\wtilde{\z}^h \\
& = & \mbf{S}\bigl(\chi(y),\left(\,^t\nabla_z\psi^{-1}\cdot\z\right)\circ\psi(y)\bigr)\,
\sum_{j=0}^nA_j\bigl(\psi(y)\bigr)\,\left(\,^t\nabla_z\psi^{-1}\cdot\z\right)\circ\psi(y) \\
& = & \left(\mbf{S}P_1\right)\bigl(\chi(y),\left(\,^t\nabla_z\psi^{-1}\cdot\z\right)\bigr)\,.
\end{eqnarray*}

Furthermore, as done in \cite{C-M} (see Section 5), the Green formula \eqref{eq:by-parts} can be transported by $\psi$.
As a byproduct of this discussion, we get that our statements, and in particular Proposition \ref{p:extension}, are invariant
by smooth change of variables.

Hence, we can reconduct the proof of Proposition \ref{p:extension} in the system of coordinates
$z=(t,x)\in\,]-t_0,t_0[\,\times\omega$ and $\zeta=(\tau,\xi)$. In particular we will assume that $\varSigma\,=\,\{t=0\}$,
whose unit normal is $dt$, and that $z_0=0$. Moreover, by Proposition 4.12 of \cite{M-2014}, up to shrink our domain we can
suppose that the full symmetrizer is positive in the direction $dt$ at any point $(t,x)$; by assumption \eqref{eq:det_pos},
we can also assume that $A_0(t,x)$ is invertible for any $(t,x)$.

In the end, we can recast our operator in the form $Lu\,=\,\d_tu\,+\,\sum_jA_j(t,x)\,\d_ju\,+\,B(t,x)u$, with existence
of a microlocal symmetrizer (in the sense of Definition \ref{d:micro_symm}) $S(t,x,\xi)\,=\,\mbf{S}(t,x,1,\xi)$.

\medbreak
This having been established, we present a regularity result analogous to Proposition \ref{p:reg_global}, which will be needed in the proof of Proposition \ref{p:extension}.
\begin{lemma} \label{l:regularity}
Let $\g$ and $s$ as in the hypotheses of Theorem \ref{t:local_e}. Let $u\in H^s\bigl(\,]0,t_0[\,\times\omega\bigr)$ such that
$Lu\in L^1\bigl([0,t_0];H^s(\omega)\bigr)$.

Then one has $u\in\mc{C}\bigl([0,t_0];H^{s-1/2}(\omega)\bigr)$. In particular, the trace $u_{|t=0}$ is well-defined in $H^{s-1/2}(\omega)$,
and the initial condition in $(C\!P)$ makes sense.
\end{lemma}

\begin{proof}
The proof goes along the lines of Lemma 2.2 of \cite{C-M}, so let us just sketch it. First of all, we restrict our attention to the case $[0,t_0]\times\R^n$,
the other one being obtained working with restrictions. In addition, we make use of the spaces 
$\mc H^{s,\sigma}\bigl([0,T_*]\times\R^n\bigr)$ of H\"ormander (see e.g. Appendix B of \cite{Horm}).

So, by hypothesis $u\in H^s\bigl([0,t_0]\times\R^n\bigr)\,=\,\mc H^{s,0}\bigl([0,t_0]\times\R^n\bigr)\,\hookrightarrow\,\mc H^{0,s}\bigl([0,t_0]\times\R^n\bigr)$.
From this, we deduce that $L_1u\,=\,\sum_jA_j\,\d_ju\,\in\,\mc H^{0,s-1}\bigl([0,t_0]\times\R^n\bigr)$. On the other hand, the same is true also for
the term $Bu$, since, by assumption, $|s-1|<\g$ (recall also Proposition \ref{p:Hol-Sob}).

Next, we notice that the hypothesis $Lu\in L^1\bigl([0,t_0];H^s(\R^n)\bigr)$ implies that $v(t)\,:=\,\int^t_0Lu(\tau)\,d\tau$
belongs to $\mc C\bigl([0,t_0];H^s(\R^n)\bigr)\,\hookrightarrow\,\mc H^{0,s}\bigl([0,t_0]\times\R^n\bigr)$.

From the previous properties, we get that $w:=u-v\,\in\,\mc H^{0,s}\bigl([0,t_0]\times\R^n\bigr)$, while its time derivative
$\d_tw\in\mc H^{0,s-1}\bigl([0,t_0]\times\R^n\bigr)$, because $\d_tw\,=\,\sum_jA_j\,\d_ju\,+\,Bu$. Therefore, by properties of the H\"ormander
spaces we get $w\in\mc H^{1,s-1}\bigl([0,t_0]\times\R^n\bigr)\,\hookrightarrow\,\mc{C}\bigl([0,t_0];H^{s-1/2}(\R^n)\bigr)$		
(see Theorem B.2.7 of \cite{Horm}). As a conclusion, also $u=w+v$ belongs to $\mc{C}\bigl([0,t_0];H^{s-1/2}(\R^n)\bigr)$.
\end{proof}

We are now ready to prove Proposition \ref{p:extension}.
\begin{proof}[Proof of Proposition \ref{p:extension}]
We focus on uniqueness first. Notice that classical Green formula forces the definition of $D_\Sigma$ to coincide with the usual trace operator for smooth functions.
Let us argue by contradiction and suppose that there exist two extensions $D^1_\Sigma$ and $D^2_\Sigma$, and a tempered distribution
$u\,\in\,\mc{D}(P)\,=\,\bigcup_{s>1-\g}\mc{D}(P;H^s)$ such that $D^1_\Sigma u\,=\,u^1_0$ and $D^2_\Sigma u\,=\,u^2_0$, with $u^1_0\neq u^2_0$.
In the light of Lemma \ref{l:duality}, we can suppose that $u\in\mc{D}(P;H^s)$, with $s\in\,]1-\g,1/2]$. Let us set $Pu=f\in H^s_{loc}(\Omega_\geq)$
and $\delta:=\left\|u^1_0-u^2_0\right\|_{H^{s-1/2}}>0$.

Fix now $\veps>0$. By density of smooth functions, we can chose smooth $g$, $v^1_0$ and $v^2_0$ such that
$$
\left\|f-g\right\|_{H^s(\Omega_\geq)}\,+\,\left\|u^1_0-v^1_0\right\|_{H^{s-1/2}(\Sigma\cap\Omega)}\,+\,\left\|u^2_0-v^2_0\right\|_{H^{s-1/2}(\Sigma\cap\Omega)}\,\leq\,\veps\,.
$$



Thanks to the previous discussion, we can work in local coordinates $(t,x)$, and hence suppose the following facts: first, that $\varSigma=\{t=0\}$, and moreover that,
in these coordinates, $P$ takes the form of $L$ as defined in \eqref{def:L}.
Therefore, we can apply Theorem \ref{t:global_e} to the initial data $v^1_0$ and $v^2_0$ and external force $g$: we find two solutions $v^1$ and $v^2$ respectively,
which belong to the space $\mc{C}\bigl([0,T_*];H^{s}(\R^n;\R^m)\bigr)$ and which solves the problems, for $i=1\,,\,2$,
$$
Lv^i\,=\,g\;,\qquad v^i_{|t=0}\,=\,v^i_0\,.
$$
Notice that, in particular, we get $v^i\in\mc{H}^{0,s}\bigl([0,T_*]\times\R^n\bigr)$, and then (by the equation) $\d_tv^i\in\mc{H}^{0,s-1}\bigl([0,T_*]\times\R^n\bigr)$, which implies
$v^i\in\mc{H}^{1,s-1}\bigl([0,T_*]\times\R^n\bigr)$ (see also the proof of Lemma \ref{l:regularity} here above). Since $s-1<0$, from this property it is easy to deduce that
$v^i\in\mc{H}^{s,0}\bigl([0,T_*]\times\R^n\bigr)$. Remark that, up to shrink $\Omega_\geq$ (i.e. take a smaller existence time $T_*$), we can also suppose
that $v^i\in\mc{H}^{\sigma,0}\bigl([0,T_*]\times\R^n\bigr)$, with $\sigma>1/2$.
Furthermore, by linearity of $L$ and Theorem \ref{t:global_e} we infer the estimates
$$
\left\|u-v^i\right\|_{H^s}\,\leq\,C(\veps)\qquad\qquad\Longrightarrow\qquad\qquad \left\|v^1-v^2\right\|_{H^s}\,\leq\,2\,C(\veps)\,.
$$
But each $v^i$ is smooth, i.e. it belongs $\mc{H}^{\sigma,0}\bigl([0,T_*]\times\R^n\bigr)$ with $\sigma>1/2$, so that $D^1_\Sigma v^i\equiv D^2_\Sigma v^i=v^i_0$ for all $i=1\,,\,2$. 
Hence, by continuity of the trace operator on smooth functions we deduce
$$
\left\|v^1_0-v^2_0\right\|_{H^{s-1/2}}\,\leq\,\left\|v^1_0-v^2_0\right\|_{H^{\sigma-1/2}}\,\leq\,C'(\veps)\,.
$$

At this point, we write
$$
\delta\,=\,\left\|u^1_0-u^2_0\right\|_{H^{s-1/2}}\,\leq\,\left\|u^1_0-v^1_0\right\|_{H^{s-1/2}}+\left\|v^1_0-v^2_0\right\|_{H^{s-1/2}}+\left\|v^2_0-u^2_0\right\|_{H^{s-1/2}}\,\leq\,
2\veps+C'(\veps)\,:
$$
taking a $\veps>0$ small enough gives the contradiction, completing in this way the proof of the uniqueness part.

Let us consider the problem of existence of the trace operator onto $\varSigma$. Once again, we work in local coordinates $z=(t,x)$, with $z_0=0$ and $\varSigma=\{t=0\}$, and the conormal
given by $\nu(t,x)\,=\,\mu(x)\,dt$, for a suitable positive function $\mu$.
By Lemma \ref{l:regularity}, we can define the trace $u_{|t=0}$ as a distribution in $H^{s-1/2}(\varSigma)$: then, it is enough to check
that, in this coordinates, Green formula in Proposition \ref{p:extension} makes sense with $D_\varSigma u\,=\,u_{|t=0}$.
But from now on the arguments are analogous to the discussion in Section 5 of \cite{C-M}: so we omit them.

The proposition is now completely proved.
\end{proof}

\subsection{Proof of local existence and uniqueness} \label{ss:proof_local}

Now, we can turn our attention to the statements about existence and uniqueness of solutions to the local Cauchy problem.
Let us start with the proof of Theorem \ref{t:local_e}: it is analogous to the one of \cite{C-M} for wave operators,
so let us just sketch it.

\begin{proof}[Proof of Theorem \ref{t:local_e}]
Up to a change of variables, we can suppose, as usual, that $z_0=0$, $\varSigma=\{t=0\}$, with normal $dt$,
and that the operator $P$, in these coordinates, assumes the  form of $L\,=\,\d_t\,+\,\sum_jA_j\,\d_j\,+\,B$.

Consider a smooth function $\Phi:\R^{1+n}_y\longrightarrow\Omega$ such that $\Phi(y)=y$ for all $y\in\Omega_1\subset\Omega$, and
$\Phi(y)=z_0$ if $|y|$ is large enough.
Such a function can be built by taking $\Omega_1$ to be a ball centered in $z_0$ and working along radial directions emanating from
$z_0$, for instance.

Changing the coefficients of $L$ according to the rule $f^\sharp(y)\,:=\,f\bigl(\Phi(y)\bigr)$, we are led to consider a new operator
$L^\sharp$, such that $L^\sharp\equiv L$ on $\Omega_1$, having coefficients with the same regularity as $L$ and admitting
a full symmetrizer $\mbf{S}^\sharp(y,\z)\,=\,\mbf{S}\bigl(\Phi(y),\z\bigr)$ which is positive in the directions $\nu\bigl(\Phi(y)\bigr)$.

Fix now a $s\in\,]1-\g,\g[\,$, and take another $s_1\in\,]1-\g,s[$. Let us apply Theorem \ref{th:en_LL} to operator $L^\sharp$: this provides
us with a loss parameter $\beta$ and with an existence time $T\,=\,(s-s_1)/\beta$. Define then $\Omega_0\,:=\,\Omega_1\cap\{|t|<T\}$
and $\omega\,:=\,\Omega_0\cap\{t=0\}$.

Let us take an initial datum $u_0\in H^s(\omega)$: by definition \eqref{def:H^s}, it can be seen as the restriction to $\omega$ of a
$u_0^\sharp\in H^s(\R^n)$. In the same way, $f\in H^s(\Omega_0\cap\{t>0\})$ is the restriction of a $f^\sharp\in H^s(\R^{1+n}\cap\{t>0\})$,
and in particular $f^\sharp\in L^2\bigl(]0,T[\,;H^s(\R^n)\bigr)$.

Then, by use of Theorem \ref{th:en_LL} we solve the Cauchy problem
$$
L^\sharp u^\sharp\,=\,f^\sharp\;,\qquad u^\sharp_{|t=0}\,=\,u_0^\sharp\,:
$$
we find a solution $u^\sharp$ on $[0,T]\times\R^n$ which, in particular, belongs to $L^2\bigl([0,T];H^{s_1}(\R^n)\bigr)$. But from this
fact together with Proposition \ref{p:Hol-Sob} and by use of the equation, we deduce that
$\d_tu^\sharp\in L^2\bigl(]0,T[\,;H^{s_1-1}(\R^n)\bigr)$,
which finally implies that $u^\sharp\in\mc H^{1,s_1-1}\bigl([0,T]\times\R^n\bigr)\,\hookrightarrow\,H^{s_1}\bigl([0,T]\times\R^n\bigr)$,
the inclusion following from the fact that $s_1-1<0$.

Therefore, by restriction we infer the existence of a solution $u$ to $(C\!P)$ in $\Omega_0$.
\end{proof}

This having been proved, we turn our attention to the question of local uniqueness.
We start by establishing a result about propagation of zero across the surface $\{t=0\}$.
\begin{lemma} \label{l:0-prop}
Let $s>1-\g$ and $\Omega=\,]-t_0,t_0[\,\times\omega$ for some $t_0>0$ and $\omega\subset\R^n$ neighborhood of $0$.
Suppose that $u\in H^s(\Omega\cap\{t>0\})$ fulfills
\begin{equation} \label{eq:0-Cauchy}
Lu\,=\,0\;,\qquad u_{|t=0}\,=\,0\,.
\end{equation}
Denote by $u_e$ the extension of $u$ by $0$ on $\{t<0\}$.

Then, there exists a neighborhood $\Omega_1\subset\Omega$ of $0$ such that $u_e\in H^s(\Omega_1)$ and $Lu_e\,=\,0$ on $\Omega_1$.
\end{lemma}

\begin{proof}
Since $u\in H^s(\Omega\cap\{t>0\})$, we have in particular that $u\in L^2\bigl([0,t_0];H^s_{loc}(\omega)\bigr)$,
and therefore $u_e$ belongs to $L^2\bigl([-t_0,t_0];H^s_{loc}(\omega)\bigr)$.

On the other hand, $\d_tu\in L^2\bigl([0,t_0];H^{s-1}_{loc}(\omega)\bigr)$, with $u_{|t=0}\,=\,0$. Hence, the (weak)
derivative $\d_tu_e$ is the extension of $\d_tu$ by $0$ for negative times (test it on functions of the form $\d_t\vphi$).
This implies $\d_tu_e\in L^2\bigl([-t_0,t_0];H^{s-1}_{loc}(\omega)\bigr)$, and therefore
$u_e\in H^s_{loc}\bigl(\,]-t_0,t_0[\,\times\omega\bigr)$ (using the same arguments as in the previous proof).

Finally, veryfing that $Lu_e=0$ is an easy matter. Indeed, testing the equations on smooth
$\vphi\in\mc{D}\bigl(\,]-t_0,t_0[\,\times\,\omega\bigr)$, the integral in time reduces on intervals of the form $\,]0,t_1[\,$,
because $u_e$ vanishes for $t<0$. But for $t>0$, $u_e\equiv u$, and $u$ is a weak solution of $Lu=0$ with $u_{|t=0}=0$:
this entails that also the integral over $\,]0,t_1[\,$ is equal to $0$.
\end{proof}

From the previous result, we can deduce the local uniqueness of solutions, by mean of classical convexification arguments.
\begin{proof}[Proof of Theorem \ref{t:local_u}]
Once again, we can fix local coordinates for which $\Omega=\,]-t_0,t_0[\,\times\omega$ and $\varSigma=\{t=0\}$.
Suppose that $u\in H^s\bigl(\Omega\cap\{t>0\}\bigr)$ satisfies \eqref{eq:0-Cauchy}. Denoting by $u_e$ its extension by $0$,
priveded by Lemma \ref{l:0-prop}, we know that $u_e$ is a $H^s$ distribution on (say) $\,]-t_1,t_1[\,\times\omega_1$, for a suitably small
$\omega_1\subset\omega$, and in the same neighborhood $Lu_e=0$.

We consider the change of variables $\psi:\,(t,x)\,\mapsto\,\left(\,\wtilde{t},\wtilde{x}\,\right)$ such that
$$
\wtilde{t}\,:=\,t\,+\,|x|^2\qquad\mbox{ and }\qquad \wtilde{x}\,:=\,x\,.
$$
Notice that this map sends $\{t<0\}$ into $\left\{\,\wtilde{t}<|\wtilde{x}|^2\,\right\}$. Let us define
$\wtilde{u}\,=\,u_e\circ\psi$ and $\wtilde{L}$ the operator obtained by $L$ under the transformation $\psi$.

Therefore, we have that $\wtilde u$ is defined in $\wtilde{t}<\wtilde{t}_1$, for a
suitable $\wtilde{t}_1>0$, and $\wtilde u\equiv 0$ in $\left\{\,\wtilde{t}<|\wtilde{x}|^2\,\right\}$.
Furthermore, up to take a smaller $\wtilde{t}_1$, we can suppose that $\wtilde{L}$ is defined on a neighborhood $\wtilde\Omega$ of the
origin which contains the closed lens $\oline{{\Theta}}\,:=\,\left\{|\wtilde{x}|^2\leq\wtilde{t}\leq\wtilde{t}_1\right\}$, and that
$\wtilde{L}\wtilde{u}\,\equiv\,0$ on $\wtilde{\Omega}$.

Now, we repeat the construction explained in the proof of Theorem \ref{t:local_e} above. Namely, we extend the coefficients
of $\wtilde{L}$ and we obtain an operator $L^\sharp$, defined on the whole $\R^{n+1}$, which preserves the regularity of the coefficients
and the microlocal symmetrizability assumption, and which coincides with $\wtilde{L}$ on a smaller neighborhood of $\oline{{\Theta}}$.

We also extend $\wtilde{u}$ to $u^\sharp$. Then, on the set $\,]-\infty,\wtilde{t}_1[\,\times\R^n$ we get
$$
L^\sharp u^\sharp\,=\,0\;,\qquad u^\sharp\in H^s\;,\qquad u^\sharp_{|\left\{\,\wtilde{t}<|\wtilde{x}|^2\,\right\}}\,=\,0\,.
$$
In particular, $u^\sharp_{|\wtilde{t}=-\delta}\,=\,0$ for $\delta>0$ arbitrarly small. Therefore energy estimates of Theorem \ref{th:en_LL},
applied to $L^\sharp$ and the initial time $-\delta$ (for $\delta$ small enough), guarantee that
$u^\sharp\equiv0$ untill a time $\wtilde{T}_0$. In particular, this is true for $\wtilde{u}$, and coming back to coordinates
$(t,x)$ implies that $u\equiv0$ in a neighborhood of the origin.
\end{proof}

{\small

 }

\end{document}